\documentclass[reqno,twoside,11pt]{amsart}
\usepackage{amsmath, amsfonts, amssymb, esint, amsthm, nicefrac, bbm, dsfont}
\usepackage{epsfig, graphicx}

\newcommand{\dx}{~\mathrm{d}x}
\newcommand{\sym}{\mathrm{sym}}

\newcommand{\Id}{\mathrm{Id}}
\newcommand{\R}{\mathbb{R}}
\def\endproof{\hspace*{\fill}\mbox{\ \rule{.1in}{.1in}}\medskip }
\newcommand*{\dbar}[1]{\bar{\bar{#1}}}

\numberwithin{equation}{section}
\theoremstyle{plain}

\newtheorem{theorem}{Theorem}[section]
\newtheorem{lemma}[theorem]{Lemma}
\newtheorem{corollary}[theorem]{Corollary}

\newtheorem{conj}[theorem]{Conjecture}

\theoremstyle{definition}

\begin{document}
\title[Improved regularity in dimension two and arbitrary
codimension]{The Monge-Amp\`ere system: \\convex integration with
  improved regularity \\ in dimension two and arbitrary codimension}
\author{Marta Lewicka}
\address{M.L.: University of Pittsburgh, Department of Mathematics, 
139 University Place, Pittsburgh, PA 15260}
\email{lewicka@pitt.edu} 

\date{\today}
\thanks{M.L. was partially supported by NSF grant DMS-2006439. 
}

\begin{abstract}
We prove a convex integration result for the
Monge-Amp\`ere system in dimension $d=2$ and arbitrary codimension $k\geq 1$.
We achieve flexibility up to the H\"older regularity $\mathcal{C}^{1,\frac{1}{1+
  4/k}}$, improving hence the previous $\mathcal{C}^{1,\frac{1}{1+ 6/k}}$
regularity that followed from flexibility up to
$\mathcal{C}^{1,\frac{1}{1+d(d+1)/k}}$ in \cite{lew_conv}, valid for
any $d,k\geq 1$. Our result agrees with flexibility up to
$\mathcal{C}^{1,\frac{1}{5}}$ for $d=2, k=1$ obtained in \cite{CS}, as
well as with the $\mathcal{C}^{1,\alpha}$ result in \cite{Kallen} where $\alpha\to 1$ as
$k\to\infty$.

\end{abstract}

\maketitle
\tableofcontents

\section{Introduction}

In this paper we present an improved convex integration result for the
following Monge-Amp\`ere system in dimension $d=2$ and arbitrary codimension $k\geq 1$:
\begin{equation}\label{MA}
\begin{split}
& \mathfrak{Det}\,\nabla^2 v= f \quad \mbox{ in }\;\omega\subset\R^2,\\
& \mbox{where }\; \mathfrak{Det}\,\nabla^2 v =
\langle \partial_{11}v, \partial_{22}v\rangle -
\big|\partial_{12} v\big|^2\quad \mbox{ for }\; v:\omega\to\R^k
\end{split}
\end{equation}
In \cite{lew_conv} we studied the operator
$\mathfrak{Det}\,\nabla^2 v$ for arbitrary $d,k\geq 1$ and proved its
{\em flexibility up to $\mathcal{C}^{\frac{1}{1+d(d+1)/k}}$}, in the sense of
Theorem \ref{th_CI_weakMA} below. For $d=2$, this means flexibility up
to $\mathcal{C}^{\frac{1}{1+6/k}}$, and the main purpose of the
present paper is to increase the H\"older exponent to have $\mathcal{C}^{\frac{1}{1+4/k}}$.
Our result also generalizes \cite[Theorem 1.1]{lewpak_MA} and \cite[Theorem 1.1]{CS},
where flexibility for (\ref{MA}) in dimensions $d=2$, $k=1$ was shown
up to $\mathcal{C}^{1,\frac{1}{7}}$ and $\mathcal{C}^{1,\frac{1}{5}}$, respectively.

\bigskip

\noindent The closely related problem is that of isometric
immersions of a given Riemannian metric $g$:
\begin{equation}\label{II}
\begin{split}
& (\nabla u)^T\nabla u = g \quad\mbox{ in }\;\omega,
\\ &  \mbox{for }\; u:\omega\to \R^{2+k}.
\end{split}
\end{equation} 
which reduces to (\ref{MA}) upon
taking the family of metrics $\{g_\epsilon=\Id_2+\epsilon A\}_{\epsilon\to 0}$,
each a small perturbation of $\Id_2$, making an ansatz $u_\epsilon =
\mbox{id}_2+ \epsilon v +\epsilon^2 w$, and gathering the lowest
order terms in the $\epsilon$-expansions. This leads to the following system:
\begin{equation}\label{VK}
\begin{split}
& \frac{1}{2}(\nabla v)^T\nabla v + \sym \nabla w = A\quad\mbox{ in }\;\omega,\\
& \mbox{for }\; v:\omega\to \R^k, \quad w:\omega\to\R^2.
\end{split}
\end{equation}
On simply connected $\omega$, the above system is then equivalent to 
$\mathfrak{Det}\,\nabla^2 v= -\mbox{curl}\,\mbox{curl}\, A$,
reflecting the agreement of the Gaussian curvatures of $g_\epsilon$ and
of the surface $u_\epsilon (\omega)$ at their lowest order terms in $\epsilon$, and 
bringing us back to (\ref{MA}).

\bigskip

\noindent The special case of $k=1$ finds application in the theory of
elasticity, where the left hand side of (\ref{VK}) represents the von
K\'arm\'an stretching content  $\frac{1}{2}\nabla v\otimes \nabla v + \sym \nabla w$,
written in terms of the scalar out-of-plane displacement $v$ and the
in-plane displacement $w$ of the mid-plane $\omega$ in a thin
film. Then, (\ref{MA}) reduces to the scalar Monge-Amp\`ere equation
$\det\nabla^2v=-\mbox{curl}\,\mbox{curl}\, A =f$,
studied in our previous work \cite{lewpak_MA}. We refer the reader to
\cite{lew_conv} for a complete discussion of the relation among
(\ref{MA}), (\ref{II}) and (\ref{VK}), in the general case of
arbitrary $d$ and $k$.

\bigskip

\noindent Our main result states that any $\mathcal{C}^1$-regular pair $(v,w)$ 
which is a subsolution of (\ref{VK}), can be uniformly approximated
by a sequence of solutions $\{(v_n,w_n)\}_{n=1}^\infty$ of regularity
$\mathcal{C}^{1,\alpha}$, as follows:

\begin{theorem}\label{th_final}
Let $\omega\subset\R^2$ be an open, bounded domain. Given
$v\in\mathcal{C}^1(\bar\omega,\R^k)$, $w\in\mathcal{C}^1(\bar\omega,\R^2)$ and
$A\in\mathcal{C}^{0,\beta}(\bar\omega,\R^{2\times
  2}_\sym)$ for some $\beta\in (0,2)$, assume that:
\begin{equation}\label{baqara}
A-\big(\frac{1}{2}(\nabla v)^T\nabla v + \sym\nabla
w\big) >c\,\Id_2 \; \mbox{ on  } \; \bar\omega,
\end{equation}
for some $c>0$, in the sense of matrix inequalities. Fix $\epsilon>0$ and let:
$$0<\alpha<\min\Big\{\frac{\beta}{2},\frac{1}{1+\frac{4}{k}}\Big\}.$$ 
Then, there exists $\tilde v\in\mathcal{C}^{1,\alpha}(\bar \omega,\R^k)$ and
$\tilde w\in\mathcal{C}^{1,\alpha}(\bar\omega,\R^2)$ such that the following holds:
%\begin{equation}\label{stage_est}
\begin{align*}
& \|\tilde v - v\|_0\leq \epsilon, \quad \|\tilde w - w\|_0\leq \epsilon,
\tag*{(\theequation)$_1$}\refstepcounter{equation} \label{FFbound1}\vspace{1mm}\\ 
& A -\big(\frac{1}{2}(\nabla \tilde v)^T\nabla \tilde v + \sym\nabla
\tilde w\big) =0 \quad \mbox{ in }\;\bar\omega.
\tag*{(\theequation)$_2$} \label{FFbound2}
\end{align*}
%\end{equation}
\end{theorem}

\noindent The above result implies, as in \cite{lew_conv}:

\begin{corollary}\label{th_CI_weakMA} 
For any $f\in L^{\infty} (\omega, \R)$ on an open, bounded, simply connected
domain $\omega\subset\mathbb{R}^2$, the following holds.
Fix $k\geq 1$ and fix an exponent:
$$0< \alpha< \frac{1}{1+\frac{4}{k}}.$$ 
Then the set of $\mathcal{\mathcal{C}}^{1,\alpha}(\bar\omega, \R^k)$
weak solutions to (\ref{MA}) is dense in $\mathcal{\mathcal{C}}^0(\bar\omega, \R^k)$. 
Namely, every $v\in \mathcal{\mathcal{C}}^0(\bar\omega,\R^k)$ is the
uniform limit of some sequence $\{v_n\in\mathcal{\mathcal{C}}^{1,\alpha}(\bar\omega,\R^k)\}_{n=1}^\infty$,
such that:
\begin{equation*} 
\mathfrak{Det}\, \nabla^2 v_n  = f \quad \mbox{ on } \; \omega\;
\mbox{ for all }\; n=1\ldots\infty.
\end{equation*}
\end{corollary} 

\bigskip

\noindent For the isometric immersion problem (\ref{II}) with
$d\geq 1$, $k=1$, corresponding to codimension $1$ im\-mer\-sions
$u:\omega\to\R^{d+1}$, a version of Theorem 
\ref{th_final} has been shown in \cite[Theorem 1.1]{CDS} with
flexibility up to $\mathcal{C}^{1,\frac{1}{1+d(d+1)}}$. The case $d=2$
is special and flexibility (still in co\-di\-men\-sion $1$) of (\ref{MA})
was proved to hold up to $\mathcal{C}^{1,\frac{1}{5}}$ in \cite{DIS1/5}
using the conformal equivalence of two-\-di\-men\-sional metrics to the
Euclidean metric, which is the fact whose linear counterpart we
utilize in the present work as well.
On the other hand, a result in \cite[Theorem 1.1]{Kallen} yields
flexibility for (\ref{II}) up to $\mathcal{C}^{1,\alpha}$ for 
$\alpha$ arbitrarily close to $1$ as $k\to\infty$, in agreement 
with our Theorem \ref{th_final} and \cite[Theorem 1.1]{lew_conv}.
We point out that the dependence of
regularity on $k$ has not been quantified in \cite{Kallen}, and that
having $k\geq d(d+1)$ was essential even for the local version of that result.

\medskip

\subsection{Overview of the paper: sections \ref{sec_step} and \ref{sec_stage}} 
We state all the intermediary results for general dimensions and
codimensions $d,k\geq 1$, and specify to $d=2$ only when necessary.

\bigskip

\noindent In section \ref{sec_step} we gather preliminary estimates and building
blocks for the proof of Theorem \ref{th_final}. First, we
recall the single ``step'' construction of the convex integration
algorithm from \cite{lew_conv}. Since now it is essential to achieve
cancellations of the one-dimensional primitive deficits with
least error possible, the previous definition of perturbation fields from
\cite{lewpak_MA} would not work for this purpose.
Second, we recall the convolution and commutator estimates
from \cite{CDS}. Third, we present a matrix decomposition result in Lemma \ref{lem_diagonal}, specific to the
present dimension $d=2$. This is essentially a reformulation of a
result in \cite{CS}, which allows us to make a conjecture for
$d\geq 3$ and the flexibility exponent that could be achieved this way.
Fourth, we recall the first step in the Nash-Kuiper iteration from
\cite{lew_conv} which decreases the positive-definite deficit arbitrarily, in
particular permitting the application of Theorem \ref{thm_stage} below.

\bigskip

\noindent In section \ref{sec_stage} we carry out the ``stage''
construction, that is the first main contribution of this paper and a
technical ingredient towards the flexibility range
stated in Theorem \ref{th_final}. Namely:

\begin{theorem}\label{thm_stage}
Given an open, bounded, smooth domain $\omega\subset \R^2$,
there exists $l_0\in (0,1)$ such that the following holds for every $l\in (0,l_0)$.
Fix $v\in\mathcal{C}^2(\bar\omega + \bar B_{2l}(0),\R^k)$,
$w\in\mathcal{C}^2(\bar\omega + \bar B_{2l}(0),\R^2)$ and
$A\in\mathcal{C}^{0,\beta}(\bar\omega + \bar B_{2l}(0),\R^{2\times
  2}_\sym)$ defined on the closed $2l$-neighbourhood of $\omega$. Further, fix
an exponent $\gamma$ and constants $\lambda, M$ satisfying:
$$\gamma\in (0,1), \quad \lambda>\frac{1}{l},\quad 
M\geq\max\{\| v\|_2 , \|w\|_2, 1\}.$$
Then, there exist $\tilde v\in\mathcal{C}^2(\bar \omega+\bar B_l(0),\R^k)$,
$\tilde w\in\mathcal{C}^2(\bar\omega+\bar B_l,\R^2)$ defined on the
closed $l$-neighbourhood of $\omega$, such that, denoting the defects:
\begin{equation*}
\begin{split}
& \mathcal{D}=A -\big(\frac{1}{2}(\nabla v)^T\nabla v + \sym\nabla
w\big), \quad \tilde{\mathcal{D}}=A -\big(\frac{1}{2}(\nabla \tilde
v)^T\nabla \tilde v + \sym\nabla \tilde w\big),
\end{split}
\end{equation*}
the following bounds hold: 
%\begin{equation}\label{stage_est}
\begin{align*}
& \hspace{-3mm} \left. \begin{array}{l} \|\tilde v - v\|_1\leq
C\lambda^{\gamma/2}\big(\|\mathcal{D}\|_0^{1/2}+lM\big), \vspace{1mm} \\ 
\|\tilde w - w\|_1\leq C\lambda^{\gamma}\big(\|\mathcal{D}\|_0^{1/2}+lM\big)
\big(1+ \|\mathcal{D}\|_0^{1/2}+lM+\|\nabla v\|_0\big), \end{array}\right.%\}
\tag*{(\theequation)$_1$}\refstepcounter{equation} \label{Abound1}\vspace{5mm}\\
& \hspace{-3mm} \left. \begin{array}{l} \|\nabla^2\tilde v\|_0\leq C{\displaystyle{\frac{(\lambda
 l)^J}{l}\lambda^{\gamma/2}}}\big(\|\mathcal{D}\|_0^{1/2}+lM\big),\vspace{1mm}\\ 
\|\nabla^2\tilde w\|_0\leq C{\displaystyle{\frac{(\lambda
  l)^J}{l}}}\lambda^{\gamma}\big(\|\mathcal{D}\|_0^{1/2}+lM\big)
\big(1+\|\mathcal{D}\|_0^{1/2}+lM+\|\nabla v\|_0\big), \end{array}\right.
\tag*{(\theequation)$_2$}\label{Abound2} \medskip\\ 
& \|\tilde{\mathcal{D}}\|_0\leq C\Big(l^\beta{\|A\|_{0,\beta}}
+\frac{1}{(\lambda l)^S} \lambda^{\gamma}\big(
\|\mathcal{D}\|_0 + (lM)^2\big)\Big). \tag*{(\theequation)$_3$} \label{Abound3}
\end{align*}
%\end{equation}
Above, the norms related to functions $v,w, A, \mathcal{D}$ and $\tilde v,
\tilde w, \tilde{\mathcal{D}}$ are taken on the respective domains of their definiteness.
The constants $C$ depend only on $\omega, k, \gamma$.
The exponents $S,J$ are given through the least common multiple
of the dimension $2$ and the codimension $k$ in: 
\begin{equation}\label{lcm_defe}
lcm(2,k)=2S=kJ.
\end{equation}
\end{theorem}

\noindent We outline how Theorem \ref{thm_stage} differs from
\cite{lewpak_MA, lew_conv}. In \cite{lewpak_MA} as in
\cite{CDS}, a ``stage'' consisted of:
$$d_*=\frac{d(d+1)}{2}=\mbox{dim}\,\R^{d\times d}_\sym$$ 
number of ``steps'', each cancelling one of the rank-one ``primitive'' deficits in the
decomposition of $\mathcal{D}$. The initially chosen frequency of
the corresponding one-dimensional perturbations was multiplied by a factor $\lambda l$ at each step, leading
to the increase of the second derivative by $(\lambda l)^{d_*}$, while the remaining error in
$\mathcal{D}$ was of order $\frac{1}{\lambda l}$. Presently, in agreement with
\cite{DIS1/5} and \cite{CS}, we first observe that by Lemma
\ref{lem_diagonal} any positive definite deficit may be replaced by a
positive multiple of $\Id_2$ modulo a symmetric gradient of a in-plane field with
controlled H\"older norms. This reduces the number of primitive
deficits from $2_*=3$ to $2$. Second, it is possible to cancel $k$ such
deficits at once, by using $k$ linearly independent
codimension directions. Since there are $2$ primitive deficits, then after
cancelling these, one may proceed to cancelling the second order deficits
obtained as in the rank-one decomposition of the error between the
original and the decreased $\mathcal{D}$; the corresponding
frequencies must be then increased by the factor $(\lambda l)^{1/2}$,
precisely due to the decrease of $\mathcal{D}$ by 
$\frac{1}{\lambda l}$, as before. We inductively proceed in this
fashion (see Figure \ref{fig1}), cancelling even higher order deficits, and adding $k$-tuples of single
codimension perturbations, for a total of
$N=lcm(2,k)$ steps. The frequencies increase by the factor of
$\lambda l$ over each multiple
of $k$, leading to the total increase of second derivatives by $(\lambda
l)^J$ in \ref{Abound2},  and by the factor of $(\lambda l)^{1/2}$ over
each multiple of $2$ (i.e. at even steps),
implying the total decrease of the deficit by the factor of $\frac{1}{(\lambda l)^S}$
in \ref{Abound3}. Third, each application of Lemma
\ref{lem_diagonal} at even steps, yields bounds in terms of the H\"older norms
with a necessarily positive $\gamma$ (due to Schauder's estimates in
that proof); hence we need to interpolate between the
previously controlled norms and higher order norms, leading to the new factor
$\lambda^\gamma$ in all estimates \ref{Abound1}-\ref{Abound3}.

\medskip

\subsection{Overview of the paper: sections \ref{sec4} to \ref{sec_appli}} 

Even though Theorem \ref{thm_stage} was specific to dimension $d=2$
due to Lemma \ref{lem_diagonal}, the Nash-Kuiper scheme involving induction
on stages may be stated more generally. Section \ref{sec4} presents the proof of:

\begin{theorem}\label{th_NashKuiHol}
Let $\omega\subset\R^d$ be an open, bounded, smooth domain 
and let $k\geq 1$, $l_0\in (0,1)$ be such that the statement of Theorem \ref{thm_stage}
holds true with some given exponents $S,J\geq 1$ (not necessarily satisfying condition
(\ref{lcm_defe})). Then we have the following.  
For every $v\in\mathcal{C}^2(\bar\omega + \bar B_{2l_0}(0),\R^k)$,
$w\in\mathcal{C}^2(\bar\omega +\bar B_{2l_0}(0),\R^d)$
and $A\in\mathcal{C}^{0,\beta}(\bar\omega +\bar B_{2l_0}(0), \R^{d\times d}_\sym)$, such that:
$$\mathcal{D}=A-\big(\frac{1}{2}(\nabla v)^T\nabla v + \sym\nabla
w\big) \quad\mbox{ satisfies } \quad 0<\|\mathcal{D}\|_0\leq 1,$$
and for every $\alpha$ in the range:
\begin{equation}\label{rangeAl}
0< \alpha <\min\Big\{\frac{\beta}{2},\frac{S}{S+2J}\Big\},
\end{equation}
there exist $\tilde v\in\mathcal{C}^{1,\alpha}(\bar\omega,\R^k)$ and
$\tilde w\in\mathcal{C}^{1,\alpha}(\bar\omega,\R^d)$ with the following properties:
\begin{align*}
& \|\tilde v - v\|_1\leq C \big(1+\|\nabla v\|_0\big)^2
\|\mathcal{D}_0\|_0^{1/4}, \quad \|\tilde w -
w\|_1\leq C(1+\|\nabla v\|_0)^3\|\mathcal{D}\|_0^{1/4},
\tag*{(\theequation)$_1$}\refstepcounter{equation} \label{Hbound1}\vspace{1mm}\\
& A-\big(\frac{1}{2}(\nabla \tilde v)^T\nabla \tilde v + \sym\nabla
\tilde w\big) =0 \quad\mbox{ in }\; \bar\omega. \tag*{(\theequation)$_2$}\label{Hbound2} 
\end{align*}
%\end{equation}
The norms in
the left hand side of \ref{Hbound1} are taken on $\bar\omega$, and in the right hand
side on $\bar\omega+ \bar B_{2l_0}(0)$. The constants $C$ depend only
on $\omega, k, A$ and $\alpha$. 
\end{theorem}

\noindent The first bound in \ref{Hbound1} is actually valid with any
power smaller than $\frac{1}{2}$ in $\|\mathcal{D}_0\|_0$, and any
power larger than $1$ or $2$ in $1+\|\nabla v_0\|_0$, in $\|\tilde v - v\|_1$
or $\|\tilde w - w\|_1$, respectively. This is consistent with
\cite[Theorem 4.1]{lew_conv}, however the presented bounds are enough for our purpose.

\bigskip

\noindent The proof of Theorem \ref{th_NashKuiHol} is quite
technical. It involves iterating Theorem \ref{thm_stage}, where the
key challenge is to choose 
the right progression of parameters $l\to 0$, $\lambda\to\infty$ and
$M\to\infty$, not only consistent with the inductive procedure assumptions
and yielding $\|\mathcal{D}\|_0\to 0$, but also to guarantee that the rate of blow-up of $\|v\|_2,
\|w\|_2$ can be compensated by the control on $\|v\|_1, \|w\|_1$,
thus admitting the control of $\|v\|_{1,\alpha}, \|w\|_{1,\alpha}$
through the interpolation inequality. We show how these choices follow
naturally, separately in the two cases of
$\frac{\beta}{2}>\frac{S}{S+2J}$ and  $\frac{\beta}{2}\leq\frac{S}{S+2J}$
and that they may be achieved with sufficiently small positive $\lambda$. As in
the iteration scheme for (\ref{II}) in \cite{DIS1/5} and for (\ref{VK}) in
\cite{CS}, both valid for $d=2, k=1$, we use the ``double
exponential'' ansatz; a technical idea borrowed from the iteration
scheme in \cite{15} where the double exponential decay was used to produce
H\"older solutions to the Euler equations. The fact that we separate estimates
Theorem \ref{th_NashKuiHol} from Theorem \ref{thm_stage} provides a
cleaner ``modular'' proof, ready to tackle the dimension $d>2$, should
a version of Conjecture \ref{conje} become available.
In section \ref{sec5}, we finally prove Theorem \ref{th_final}, which
at this point becomes quite straightforward. 

\bigskip

\noindent In the last section \ref{sec_appli} we present an application of Theorem
\ref{th_final} for deriving the scaling laws bounds of the so-called
prestrained elastic energies of thin films, in the context of the quantitative isometric immersion problem.
We only recall the related setup and state the result, since the proof is exactly
the same as in \cite[Theorem 7.1]{lew_conv}.

\medskip

\subsection{Notation.}
By $\mathbb{R}^{d\times d}_{\sym}$ we denote the space of symmetric
$d\times d$ matrices. %and by  $\mathbb{R}^{d\times d}_{\sym, >}$ we
%denote the cone of symmetric, positive definite $d\times d$ matrices.
The space of H\"older continuous vector fields
$\mathcal{C}^{m,\alpha}(\bar\omega,\R^k)$ consists of restrictions of
all $f\in \mathcal{C}^{m,\alpha}(\mathbb{R}^d,\R^k)$ to the closure of
an open, bounded domain
$\omega\subset\R^d$. Then, the $\mathcal{C}^m(\bar\omega,\R^k)$ norm of such restriction is
denoted by $\|f\|_m$, while its H\"older norm in $\mathcal{C}^{m, \alpha}(\bar\omega,\R^k)$ is $\|f\|_{m,\alpha}$.
By $C>0$ we denote a universal constant which may change from line to
line, but which is bigger than $1$ and independent of all parameters, unless indicated otherwise.

\section{Convex integration: the basic ``step'' and preparatory statements}\label{sec_step}

The following single ``step'' construction, see \cite[Lemma 2.1, Corollary 2.2]{lew_conv}, is a building block of
the convex integration algorithm in this paper. We recall that a similar calculation in \cite{lewpak_MA}
based on \cite{CDS}, had $\bar\Gamma=0$ in the formula below, resulting in the presence of the extra term
$-\frac{2}{\lambda} a \dbar\Gamma(\lambda t_\eta)\sym(\nabla a \otimes
\eta)$ in the right hand side of (\ref{step_err}). With that term,
achieving the error bounds in Theorem \ref{thm_stage} would not be possible.
Namely, we have:

\begin{lemma}\label{lem_step}
Let $v\in \mathcal{C}^2(\R^d, \R^{k})$, $w\in \mathcal{C}^1(\R^d,
\R^{d})$ and $a\in \mathcal{C}^2(\R^d,\R)$ be given. Denote:
$$\Gamma(t) = 2\sin t,\quad \bar\Gamma(t) = -\frac{1}{2}\cos (2t), \quad
\dbar\Gamma(t) = -\frac{1}{2}\sin (2t),$$ 
and for two unit vectors $\eta\in\R^d$, $E\in \R^k$ and a frequency
$\lambda>0$, define:
\begin{equation}\label{defi_per}
\begin{split}
&\tilde v(x) = v(x) + \frac{1}{\lambda}a(x) \Gamma(\lambda t_\eta)E\\
& \tilde w(x) = w(x) -\frac{1}{\lambda}a(x) \Gamma(\lambda t_\eta)\nabla \langle v(x), E\rangle - 
\frac{1}{\lambda^2} a(x) \bar\Gamma(\lambda t_\eta)\nabla a(x)
+ \frac{1}{\lambda}a(x)^2 \dbar\Gamma(\lambda t_\eta)\eta.
\end{split}
\end{equation}
where $t_\eta = \langle x,\eta\rangle$. Then, the following identity is valid on $\R^d$:
\begin{equation}\label{step_err}
\begin{split}
& \big(\frac{1}{2}(\nabla \tilde v)^T \nabla \tilde v + \sym\nabla \tilde w\big) - 
\big(\frac{1}{2}(\nabla v)^T \nabla v + \sym\nabla w\big) - a^2\eta\otimes\eta 
\\ & = -\frac{1}{\lambda} a \Gamma(\lambda t_\eta)\nabla^2 \langle v, E\rangle +
\frac{1}{\lambda^2}\big(\frac{1}{2}\Gamma(\lambda
t_\eta)^2-\bar\Gamma(\lambda t_\eta)\big) \nabla a\otimes\nabla a -
\frac{1}{\lambda^2}a \bar\Gamma(\lambda t_\eta)\nabla^2a.
\end{split}
\end{equation}
%where $\frac{1}{2}\Gamma(t)^2 -\bar\Gamma(t) = 1-\frac{1}{2}\cos(2t)$.
\end{lemma}

\smallskip

\noindent As pointed out in \cite{lew_conv}, taking several
perturbations in $\tilde v$ of the form 
$\frac{1}{\lambda}a_i\Gamma(\lambda t_{\eta_i}) E_i$ corresponding to
the mutually orthogonal directions $\{E_i\}_{i=1}k$ and matching them with 
perturbations in $\tilde w$ as in (\ref{defi_per}), achieves
cancellation of $k$ nonnegative primitive deficits of the form $a_i^2\eta_i\otimes\eta_i$
while the errors in (\ref{step_err}) accumulate in a linear
fashion. This is how we use the larger codimension $k$ to increase the
H\"older regularity in Theorem \ref{th_final}.

\bigskip

\noindent We will frequently call on the convolution and commutator estimates \cite[Lemma 2.1]{CDS}: 

\begin{lemma}\label{lem_stima}
Let $\phi\in\mathcal{C}_c^\infty(\R^d,\mathbb{R})$ be a standard
mollifier that is nonnegative, radially symmetric, supported on the
unit ball $B(0,1)\subset\R^d$ and such that $\int_{\mathbb{R}^d} \phi \dx = 1$. Denote: 
$$\phi_l (x) = \frac{1}{l^d}\phi(\frac{x}{l})\quad\mbox{ for all
}\; l\in (0,1], \;  x\in\R^d.$$
Then, for every $f,g\in\mathcal{C}^0(\mathbb{R}^d,\R)$ and every
$m\geq 0$, $\beta\in (0,1]$, there holds:
\begin{align*}
& \|\nabla^{(m)}(f\ast\phi_l)\|_{0} \leq
\frac{C}{l^m}\|f\|_0,\tag*{(\theequation)$_1$}\vspace{1mm} \refstepcounter{equation} \label{stima1}\\
& \|f - f\ast\phi_l\|_0\leq C \min\big\{l^2\|\nabla^{2}f\|_0,
l\|\nabla f\|_0, {l^\beta}\|f\|_{0,\beta}\big\},\tag*{(\theequation)$_2$} \vspace{1mm} \label{stima2}\\
& \|\nabla^{(m)}\big((fg)\ast\phi_l - (f\ast\phi_l)
(g\ast\phi_l)\big)\|_0\leq {C}{l^{2- m}}\|\nabla f\|_{0} \|\nabla g\|_{0}, \tag*{(\theequation)$_3$} \label{stima4}
\end{align*}
with a constant $C>0$ depending only on the differentiability exponent $m$.
\end{lemma}

\medskip

The next auxiliary result is specific to dimension $d=2$. We reformulate \cite[Proposition 3.1]{CS}: 

\begin{lemma}\label{lem_diagonal}
Let $\omega\subset\R^2$ be an open, bounded and Lipschitz set. There exist maps: 
$$\bar\Psi: L^2(\omega,\R^{2\times
  2}_\sym)\to W^{1,2}(\omega,\R^2), \qquad \bar a: L^2(\omega,\R^{2\times  2}_\sym)\to L^{2}(\omega,\R), $$ 
which are linear, continuous, and such that:
\begin{itemize}
\item[(i)] for all $D\in L^2(\omega,\R^{2\times 2}_\sym)$ there holds:
  $D+ \sym\nabla \big(\bar\Psi(D)\big) = \bar a(D)\Id_2$,\vspace{1mm}
\item[(ii)] $\bar\Psi(\Id_2) \equiv 0$ and $\bar a(\Id_2) \equiv 1$ in
  $\omega$, \vspace{1mm}
\item[(iii)] for all $m\geq 0$ and $\gamma\in (0,1]$, if $\omega$ is
  $\mathcal{C}^{m+2,\gamma}$ regular then the maps $\bar\Psi$ and $\bar a$ are continuous from
  $\mathcal{C}^{m,\gamma}(\bar\omega,\R^{2\times 2}_\sym)$ to
  $\mathcal{C}^{m+1,\gamma}(\bar\omega, \R^2)$ and
  to $\mathcal{C}^{m,\gamma}(\bar\omega, \R)$, respectively, so that:
\begin{equation}\label{diag_bounds}
\|\bar\Psi (D)\|_{m+1,\gamma}\leq C \|D\|_{m,\gamma} \mbox{ and }
~ \|\bar a (D)\|_{m,\gamma}\leq C \|D\|_{m,\gamma} \quad \mbox{ for all
}\; D\in L^2(\omega,\R^{2\times 2}_\sym).
\end{equation}
\end{itemize} 
The constants $C$ above depend on $\omega$, $m, \gamma$ but not on
$D$. Also, there exists $l_0>0$ depending only on $\omega$, such that
(\ref{diag_bounds}) are uniform on 
the closed $l$-neighbourhoods $\{\bar\omega+ \bar B_l(0)\}_{l\in (0,l_0)}$ of $\omega$.
\end{lemma}
\begin{proof}
Given $D\in L^2(\omega,\R^{2\times 2}_\sym)$, we define:
$$ \bar\Psi(D) = \big(-\partial_1\psi_1
- \partial_2\psi_2, \partial_2\psi_1 - \partial_1\psi_2\big), \qquad
\bar a(D) = D_{11}+\partial_1 \bar \Psi^1(D),$$
where $\psi_1, \psi_2$ are solutions to the following two Dirichlet problems on $\omega$:
\begin{equation}\label{sti_help}
\left\{\begin{array}{ll} \Delta\psi_1 = D_{11}-D_{22} & \mbox{ in }\omega\\
\psi_1 = 0 & \mbox{ on } \partial\omega,\end{array}\right.
\qquad \qquad
\left\{\begin{array}{ll} \Delta\psi_2 = 2D_{12} & \mbox{ in }\omega\\
\psi_2 = 0 & \mbox{ on } \partial\omega.\end{array}\right.
\end{equation}
It is clear that the maps $\bar\Psi$ and $\bar a$ are linear and satisfy (ii)
and (iii). To check condition (i), we calculate components of the symmetric
matrix field $\bar a \Id_2 - \sym\nabla \bar\Psi$:
\begin{equation*}
\begin{split}
& ~\bar a - \partial_1\bar\Psi^1 = D_{11},\\
& ~\bar a - \partial_1\bar\Psi^2 = D_{11} + \partial_1\bar\Psi^1
-\partial_2\bar\Psi^2 = D_{11}+ (-\partial_{11}\psi_1
- \partial_{12}\psi_1) - (\partial_{22}\psi_1 -\partial_{12}\psi_2)
\\ & \hspace{1.65cm}= D_{11} - \Delta\psi_1 = D_{22}, \\
& -\frac{1}{2}(\partial_1\bar\Psi^2+\partial_2\bar\Psi^1) =
-\frac{1}{2}\big(\partial_{12}\psi_1 - \partial_{11}\psi_2
-\partial_{12}\psi_1 - \partial_{22}\psi_2\big)=\frac{1}{2}\Delta\psi_2 = D_{12}. 
\end{split}
\end{equation*}
This completes the proof of (i).
The uniformity of the bounds in (\ref{diag_bounds}) follow from
the uniformity of the classical Schauder estimates for
solutions to (\ref{sti_help}). The proof is done.
\end{proof}

\smallskip

\noindent We remark that for a general dimension $d\geq 2$, carrying
out the same approach as presented in this paper would necessitate validating the following: 

\begin{conj}\label{conje}
Let $\omega\subset\R^d$ be an open, bounded, sufficiently regular
set. Then, there exist a linear proper subspace $E_d\varsubsetneqq\R^{d\times d}_\sym$ and linear maps:
$$ \bar\Psi:\mathcal{C}^{m,\gamma}(\bar\omega,\R^{d\times d}_\sym)\to 
\mathcal{C}^{m+1,\gamma}(\bar\omega,\R^{d}),\qquad 
\bar A:\mathcal{C}^{m,\gamma}(\bar\omega,\R^{d\times d}_\sym)\to 
\mathcal{C}^{m,\gamma}(\bar\omega,E_d),$$
continuous for all $m\geq 0$ and $\gamma\in (0,1]$, and such that:
\begin{itemize}
\item[(i)] for all $D\in \mathcal{C}^{m,\gamma}(\bar\omega,\R^{d\times d}_\sym)$ there holds:
  $D+ \sym\nabla \big(\bar\Psi(D)\big) = \bar A(D)$,\vspace{1mm}
\item[(ii)] $\bar\Psi(\Id_d) \equiv 0$ and $\bar A(\Id_d) \equiv \Id_d$ in $\omega$.
\end{itemize} 
\end{conj}

\noindent Indeed, Conjecture \ref{conje} would  imply flexibility of (\ref{VK})
and of the Monge-Ampere system (\ref{MA}) up to regularity
$\mathcal{C}^{1, \frac{1}{1+2(\mathrm{dim}\,E_d)/k}}$.
Lemma \ref{lem_diagonal} validates Conjecture \ref{conje} for $d=2$ and with:
$$E_2=\big\{\alpha \Id_2;~\alpha\in\R\big\},$$ 
reflecting the fact that every $2$-dimensional Riemann
metric is conformally equivalent to the Euclidean metric. 
For $d=3$, it is natural to ask if Conjecture \ref{conje} holds with the space:
$$E_3=\big\{\sum_{i=1}^3\alpha_ie_i\otimes e_i;~ \alpha_1,\alpha_2,
\alpha_3\in \R\big\}$$ 
consisting of diagonal matrices, which is motivated by and in agreement
with the fact that every $3$-dimensional metric is locally
diagonalizable. This result, without the H\"older norms estimates, may
be proved in the analytic class by an application of the
Cartan-K\"ahler theorem, and in the smooth class by a direct inspection.
For $d\geq 4$, one expects the optimal dimension:
$$\mbox{dim}\, E_d=\mbox{dim}\, \R^{d\times d}_\sym - d=
\frac{d(d-1)}{2}.$$ 
With the above, Conjecture \ref{conje} would imply flexibility up to regularity
$\mathcal{C}^{1, \frac{1}{1+d(d-1)/k}}$, while we recall that the best exponent
known at present, from \cite{lew_conv}, is $\mathcal{C}^{1, \frac{1}{1+d(d+1)/k}}$.

\bigskip

\noindent As the final preparatory result, we recall the ``first step''
in the Nash-Kuiper iteration, al\-lo\-wing to bring the sup-norm of
the given positive definite deficit, below a threshold
needed for an application of Theorem \ref{th_NashKuiHol}.
This result is independent from the H\"older continuity
estimates; its proof only necessitates the decomposition of symmetric positive definite matrices
which are close to $\Id_d$, into ``primitive metrics'' \cite[Lemma 5.2]{CDS}.
Namely, from \cite[Theorem 5.2]{lew_conv} we quote:

\begin{lemma}\label{th_approx_nonlocal}
Let $\omega\subset\R^d$ be an open, bounded set. Given
$v\in\mathcal{C}^\infty(\bar\omega,\R^k)$, $w\in\mathcal{C}^\infty(\bar\omega,\R^d)$ and
$A\in\mathcal{C}^\infty(\bar\omega,\R^{d\times d}_\sym)$, assume that:
$$\mathcal{D}=A-\big(\frac{1}{2}(\nabla v)^T\nabla v + \sym\nabla
w\big) \quad\mbox{ satisfies } \quad \mathcal{D}>c\,\Id_d \; \mbox{ on }
\; \bar\omega$$
for some $c>0$, in the sense of matrix inequalities. Fix
$\epsilon>0$. Then, there exist $\tilde v\in\mathcal{C}^\infty(\bar
\omega,\R^k)$, $\tilde w\in\mathcal{C}^\infty(\bar\omega,\R^d)$ such that
the following holds with constants $C$ depending on $d, k$ and $\omega$:
%\begin{equation}\label{stage_est}
\begin{align*}
& \|\tilde v - v\|_0\leq \epsilon, \quad \|\tilde w - w\|_0\leq \epsilon,
\tag*{(\theequation)$_1$}\refstepcounter{equation} \label{Cbound1}\vspace{1mm}\\ 
& \|\nabla (\tilde v-v)\|_0\leq C \|\mathcal{D}\|_0^{1/2}, \quad \|\nabla(\tilde w -
w)\|_0\leq C\|\mathcal{D}\|_0^{1/2}\big(\|\mathcal{D}\|_0^{1/2} +\|\nabla v\|_0\big),
\tag*{(\theequation)$_2$} \label{Cbound2}\vspace{1mm}\\
& \|A -\big(\frac{1}{2}(\nabla \tilde v)^T\nabla \tilde v + \sym\nabla
\tilde w\big)\|_0\leq \epsilon. \tag*{(\theequation)$_3$} \label{Cbound3}
\end{align*}
%\end{equation}
\end{lemma}

\section{The ``stage'' for the $\mathcal{C}^{1,\alpha}$
  approximations: a proof of Theorem \ref{thm_stage}}\label{sec_stage}

The proof consists of several steps. The inductive construction
below is a refinement of \cite[Theorem 1.2]{lew_conv} in view of Lemma \ref{lem_diagonal}, 
allowing to decrease the number of primitive metrics in the deficit decomposition
from $2_*=3$ to $2$. Recall that all constants $C$, which may change from line to line of a calculation, are 
assumed to be larger than $1$, and they depend only on $\omega$, $k$,
$\gamma$ and the differentiability exponent $m$, whenever present. 

\bigskip

\noindent {\bf Proof of Theorem \ref{thm_stage}}

\smallskip

{\bf 1. (Preparing the data)} Let $l_0$ be as in Lemma
\ref{lem_diagonal} and fix $l<l_0$. Taking $\phi_l$
as in Lemma \ref{lem_stima}, we define the following smoothed data functions on
the $l$-thickened set $\bar\omega+\bar B_l(0)$:
$$v_0=v\ast \phi_l,\quad w_0=w\ast \phi_l, \quad A_0=A\ast \phi_l,
\quad {\mathcal{D}}_0= \big(\frac{1}{2}(\nabla v_0)^T\nabla v_0 + \sym\nabla w_0\big) - A_0.$$
From Lemma \ref{lem_stima}, one deduces the initial bounds: 
% where in line with our convention, constants $C$ depend only on $m$:
\begin{align*}
& \|v_0-v\|_1 + \|w_0-w\|_1 \leq C lM,
\tag*{(\theequation)$_1$}\refstepcounter{equation} \label{pr_stima1}\\
& \|A_0-A\|_0 \leq Cl^\beta\|A\|_{0,\beta}, \tag*{(\theequation)$_2$} \label{pr_stima2}\\
& \|\nabla^{(m+1)}v_0\|_0 + \|\nabla^{(m+1)}w_0\|_0\leq
\frac{C}{l^m} lM\quad \mbox{ for all }\; m\geq 1, \tag*{(\theequation)$_3$} \label{pr_stima3}\\
& \|\nabla^{(m)} \mathcal{D}_0\|_0\leq
\frac{C}{l^m} \big(\|\mathcal{D}\|_0 + (lM)^2\big)\quad \mbox{ for
  all }\; m\geq 0. \tag*{(\theequation)$_4$}\label{pr_stima4} 
\end{align*}
Indeed,  \ref{pr_stima1}, \ref{pr_stima2} follow from \ref{stima2} and
in view of the lower bound
on $M$. Similarly, \ref{pr_stima3} follows by applying \ref{stima1} to
$\nabla^2v$ and $\nabla^2w$ with the differentiability exponent $m-1$.
Since:
$$\mathcal{D}_0 = \frac{1}{2}\big((\nabla
v_0)^T\nabla v_0 - ((\nabla v)^T\nabla v)\ast\phi_l\big) -
\mathcal{D}\ast \phi_l,  $$ 
we get \ref{pr_stima4} by applying \ref{stima1} to $\mathcal{D}$, and
\ref{stima4} to $\nabla v$.

\medskip

{\bf 2. (Induction definition: frequencies)}  We now inductively define the main
coefficients, frequencies and corrections in the construction of ($\tilde
v,\tilde w)$ from $(v,w)$. First, recall that:
\begin{equation}\label{lcm_def}
N \doteq lcm(2,k) = 2S = kJ, \qquad S,J\geq 1.
\end{equation}
We set the initial perturbation frequencies as:
$$\lambda_0 = \frac{1}{l}, \qquad \lambda_1=\lambda,$$
while for $i=2\ldots N$ we define, for $j=0\ldots J-1$ and $s=0\ldots S-1$:
\begin{equation}\label{count_lam}
\lambda_i l = (\lambda l)^{1+j+s/2 }\quad \mbox{ for all }\; i\in (kj, k(j+1)]\cap (2s, 2(s+1)].
\end{equation}
\begin{figure}[htbp]
\centering
\includegraphics[scale=0.55]{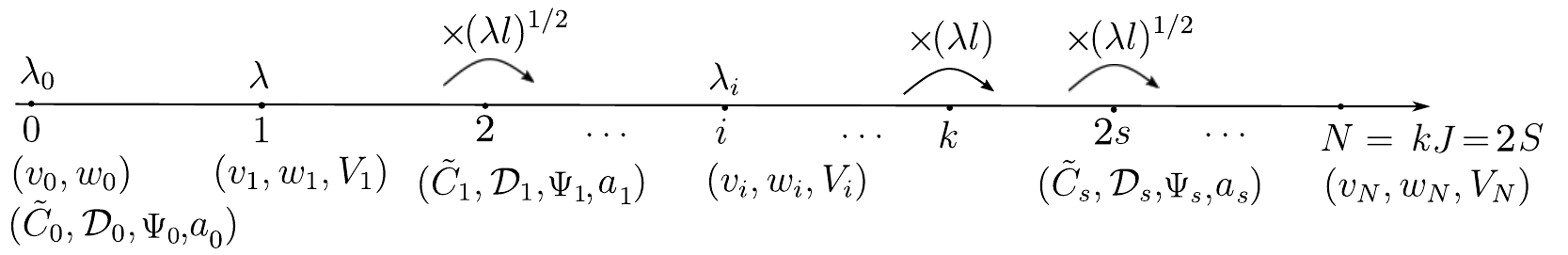}
\caption{{Progression of frequencies $\lambda_i$ and other
    intermediary quantities defined at integers $i=1\ldots N$, where $N= lcm(2,k)$.}}
\label{fig1}
\end{figure}

\medskip

{\bf 3. (Induction definition: decomposition of deficits)} For $s=0\ldots
S-1$ we define constants $\tilde C_s$, perturbation amplitudes
$a_s \in\mathcal{C}^\infty(\bar\omega+\bar B_l(0), \R)$ 
and correction fields $\Psi_s\in \mathcal{C}^\infty(\bar\omega+\bar B_l(0),
\R^2)$, by applying Lemma \ref{lem_diagonal} to the already derived deficit
$\mathcal{D}_s$ on the set $\bar\omega+\bar B_l(0)$:
\begin{equation*}
\begin{split}
& \tilde C_s = \frac{2}{r_0}\Big(\|\mathcal{D}_s\|_{0,\gamma}+\frac{(\lambda_0\lambda_2\ldots
  \lambda_{2s})^\gamma}{(\lambda l)^s} (\|\mathcal{D}\|_0+ (lM)^2)\Big),\\ 
& a_s = \big(\tilde C_s - \bar a(\mathcal{D}_s)\big)^{1/2}, \qquad 
\Psi_s = \tilde C_s id_2 - \bar\Psi(\mathcal{D}_s).
\end{split}
\end{equation*}
Above, $r_0=r_0(\gamma)>0$ is given through the requirement:
$$\bar a(D) >\frac{1}{2} ~\mbox{ on }~\bar\omega+\bar B_l(0)\quad
\mbox{ whenever } \quad \|D-\Id_2\|_{0,\gamma}<r_0,$$
whose validity is justified by Lemma \ref{lem_diagonal}. Note that
our definition of $a_s$ is correctly posed, because
$\tilde C_s - \bar a(\mathcal{D}_s) = \tilde C_s \bar a\big(\Id_2-\frac{1}
{\tilde C_s}\mathcal{D}_s\big)>0$ in view of $\|\Id_2 - (\Id_2-
\frac{1}{\tilde C_s}\mathcal{D}_s)\|_{0,\gamma}<r_0$. 
Further:
\begin{equation}\label{low_bd_as}
\mathcal{D}_s= \sym\nabla \Psi_s - a_s^2 \Id_2 \quad\mbox{ and }\quad a_s>
\Big(\frac{\tilde C_s}{2}\Big)^{1/2} \;\mbox{ in } \;\bar\omega+\bar B_l(0).
\end{equation}
We also obtain, directly from Lemma \ref{lem_diagonal}: %with constants $C$ depending on $m$ and $\gamma$:
\begin{equation}\label{as0}
\begin{split}
& \|\Psi_s\|_{m+1} \leq
C\big(\tilde C_s + \|\mathcal{D}_s\|_{m,\gamma}\big)\quad \mbox{ for all }\; m\geq 0,\\
& \|a_s\|_0\leq C\| \tilde C_s \Id_d -
\mathcal{D}_s\|_{0,\gamma}^{1/2}\leq C \tilde C_s^{1/2}.
\end{split}
\end{equation}
For the future estimate of derivatives of $a_s$ of order $m\geq 1$, we
use Fa\'a di Bruno formula's in:
\begin{equation}\label{asm}
\begin{split}
\|\nabla^{(m)}a_s\|_0 & \leq C \Big\|\sum_{p_1+2p_2+\ldots
  mp_m=m} a_s^{2(1/2-p_1-\ldots -p_m)}\prod_{t=1}^m \big|\nabla^{(t)}a_s^2\big|^{p_t}\Big\|_0\\
& \leq C \sum_{p_1+2p_2+\ldots mp_m=m}\frac{1}{\tilde C_s^{(p_1+\ldots+p_m)-1/2}}
\prod_{t=1}^m \big( \tilde C_s + \|\mathcal{D}_s\|_{t,\gamma}\big)^{p_t} \\ & \leq {C}{\tilde
  C_s^{1/2}} \sum_{p_1+2p_2+\ldots mp_m=m}\prod_{t=1}^m
\Big(1+\frac{\|\mathcal{D}_s\|_{t,\gamma}}{\tilde C_s}\Big)^{p_t},
\end{split}
\end{equation}
in virtue of the lower bound in (\ref{low_bd_as}). 

\medskip

{\bf 4. (Induction definition: perturbations)} For each $i=1\ldots N$ we uniquely write:
\begin{equation}\label{isj}
\begin{split}
i=kj+\gamma = 2s + \delta \quad \mbox{ with } \quad  & j=0\ldots J-1,
\quad \gamma=1\ldots k, \\ &  s=0\ldots S-1, \quad \delta=1,2.
\end{split}
\end{equation}
Define $v_{i}\in\mathcal{C}^\infty(\bar\omega+\bar B_l(0),\R^k)$ and
$w_i\in\mathcal{C}^\infty(\bar\omega+\bar B_l(0), \R^2)$ according to the ``step''
construction in Lemma \ref{lem_step}, involving the periodic profile functions
$\Gamma, \bar\Gamma, \dbar\Gamma$ and the notation $t_\eta=\langle x, \eta\rangle$:
\begin{equation*}
\begin{split}
& v_i(x) = v_{i-1}(x) + \frac{1}{\lambda_i} a_s(x)\Gamma(\lambda_i t_{e_\delta})e_\gamma,\\
& w_i(x) = w_{i-1}(x) - \frac{1}{\lambda_i} a_s(x)\Gamma(\lambda_it_{e_\delta})\nabla v_{i-1}^\gamma 
-  \frac{1}{\lambda_i^2} a_s(x)\bar\Gamma(\lambda_it_{e_\delta})\nabla a_s
+  \frac{1}{\lambda_i} a_s(x)^2\dbar\Gamma(\lambda_it_{e_\delta})e_\delta.
\end{split}
\end{equation*}
We observe that by construction of $v_i$, the second term in $w_{i}$ can be rewritten as follows:
\begin{equation}\label{w_simp}
\frac{1}{\lambda_i} a_s(x)\Gamma(\lambda_it_{e_\delta})\nabla v_{i-1}^\gamma 
= \frac{1}{\lambda_i} a_s(x)\Gamma(\lambda_it_{e_\delta})\nabla v_{jk}^\gamma. 
\end{equation}
We eventually set:
\begin{equation}\label{vw_fin}
\tilde v = v_N,\qquad \tilde w = w_N-\sum_{s=0}^{S-1}\Psi_s.
\end{equation}

\medskip

{\bf 5. (Induction definition: deficits)} For each $i=1\ldots N$, we define the partial deficit:
$$ V_i = \big(\frac{1}{2}(\nabla v_i)^T\nabla v_i+\sym\nabla
w_i\big) - \big(\frac{1}{2}(\nabla v_{i-1})^T\nabla v_{i-1}+\sym\nabla w_{i-1}\big),$$
and for each $s=1\ldots S$ we set the combined deficit:
$\mathcal{D}_s\in\mathcal{C}^\infty(\bar\omega +\bar B_l(0), \R^{2\times 2}_\sym)$ in:
\begin{equation*}
\begin{split}
\mathcal{D}_s  = & \; \big(\frac{1}{2}(\nabla v_{2s})^T\nabla v_{2s}+\sym\nabla
w_{2s}\big) - \big(\frac{1}{2}(\nabla v_{2(s-1)})^T\nabla v_{2(s-1)}+\sym\nabla w_{2(s-1)}\big)
-a_{s-1}^2\Id_2
\\ = & \sum_{i=2s-1}^{2s} \Big(V_i - a_{s-1}^2e_{\delta}\otimes
e_\delta\Big) = V_{2s-1} + V_{2s} - a_{s-1}^2\Id_2.
\end{split}
\end{equation*}
Above, components of the last sum we used the convention (\ref{isj}), where $\delta=\delta(i)=1,2$.
By Lemma \ref{lem_step} and (\ref{w_simp}), and setting $j=0\ldots
J-1$ again according to (\ref{isj}), we get:
\begin{equation}\label{prep_defi_s}
\begin{split}
V_i - a_{s-1}^2e_\delta\otimes e_\delta
 = & - \frac{1}{\lambda_i} a_{s-1} \Gamma(\lambda_i t_{e_\delta})\nabla^2 v_{jk}^\gamma
- \frac{1}{\lambda_i^2}a_{s-1} \bar\Gamma(\lambda_i t_{e_\delta})\nabla^2a_{s-1}
\\ & + \frac{1}{\lambda_i^2}\big(\frac{1}{2}\Gamma(\lambda_i
t_{e_\delta})^2-\bar\Gamma(\lambda_i t_{e_\delta})\big) \nabla a_{s-1}\otimes\nabla a_{s-1}.
\end{split}
\end{equation}
We right away note that, by (\ref{low_bd_as}) there holds:
\begin{equation}\label{pomag}
\begin{split}
\tilde{\mathcal{D}} & = (A-A_0) -\mathcal{D}_0 -
\Big(\big(\frac{1}{2}(\nabla\tilde v)^T\nabla \tilde v + \sym\nabla \tilde w\big)
- \big(\frac{1}{2}(\nabla v_0)^T\nabla v_0 + \sym\nabla w_0\big)\Big) \\ & = 
(A-A_0) -\mathcal{D}_0 +\sum_{s=0}^{S-1}\sym\nabla\Psi_s - \sum_{s=1}^{S}\sum_{i=2s-1}^{2s}V_i
\\ & = (A-A_0) +\sum_{s=0}^{S-1}\sym\nabla\Psi_s  - \sum_{s=0}^S\mathcal{D}_s -
\sum_{s=1}^{S}a_{s-1}^2\Id_2 = (A-A_0) -\mathcal{D}_S.
\end{split}
\end{equation}

\medskip

{\bf 6. (Inductive estimates)} In steps 7-8 below we will show the
following estimates, valid for all $m\geq -1 $ and $i=1\ldots N$,
and where $s=s(i)$ is given according to (\ref{isj}): 
%with constants $C$ that depend only on $m$, $\gamma$: 
\begin{align*}
& \hspace{-3mm} \left. \begin{array}{l} \|\nabla^{(m+1)}(v_i - v_{i-1})\|_0\leq
    C\displaystyle{ \frac{\lambda_i^m}{(\lambda l)^{s/2}}}\big(\lambda_0\lambda_2\ldots
    \lambda_{2s}\big)^{\gamma/2}\big(\|\mathcal{D}\|_0^{1/2}+lM\big),  \vspace{1mm} \\ 
\|\nabla^{(m+1)}(w_i -w_{i-1})\|_0\leq
C\displaystyle{\frac{\lambda_i^m}{(\lambda l)^{s/2}}}\big(\lambda_0\lambda_2\ldots 
    \lambda_{2s}\big)^{\gamma}\big(\|\mathcal{D}\|_0^{1/2}+lM\big) \times \\
\qquad\qquad\qquad\qquad \qquad\qquad\qquad \qquad \qquad\qquad 
\times \big(\|\mathcal{D}\|_0^{1/2}+lM+\|\nabla v\|_0\big),
\end{array}\right\}
\tag*{(\theequation)$_1$}\refstepcounter{equation} \label{Fbound1}
\end{align*}
Also, for all $m\geq 0$ and $s=0\ldots S$ we will prove that:
\begin{align*}
& \hspace{-3mm} \|\mathcal{D}_s\|_m\leq C\frac{\lambda_{2s}^m}{(\lambda l)^s}
\big(\lambda_0\lambda_2\ldots \lambda_{2(s-1)}\big)^{\gamma} \big(\|\mathcal{D}\|_0+(lM)^2\big).
\tag*{(\theequation)$_2$}\label{Fbound2} 
\end{align*}
Note that the bound \ref{Fbound2} at its
lowest counter value $s=0$, follows in view of \ref{pr_stima4}, and since 
$\lambda_0 = \frac{1}{l}$.
We further observe that, using interpolation and the preparatory bound
(\ref{asm}), the estimate \ref{Fbound2} easily implies for all $m\geq 0$ and $s=0\ldots S-1$:
\begin{align*} 
& \hspace{-3mm} \left. \begin{array}{l} \tilde C_s\leq C\displaystyle\frac{1}{(\lambda l)^s} 
\big(\lambda_0\lambda_2\ldots \lambda_{2s}\big)^{\gamma} \big(\|\mathcal{D}\|_0+(lM)^2\big), \vspace{1mm} \\
\|\Psi_s\|_{m+1}\leq  C\displaystyle{ \frac{\lambda_{2s}^m}{(\lambda l)^{s}}}\big(\lambda_0\lambda_2\ldots
    \lambda_{2s}\big)^{\gamma}\big(\|\mathcal{D}\|_0+(lM)^2\big),  \vspace{1mm} \\
 \|a_s\|_m\leq \displaystyle {C\frac{\lambda_{2s}^m}{(\lambda l)^{s/2}}}\big(\lambda_0\lambda_2\ldots
    \lambda_{2s}\big)^{\gamma/2}\big(\|\mathcal{D}\|_0^{1/2}+ lM\big).
\end{array}\right\} 
\tag*{(\theequation)$_3$}\label{Fbound3}
\end{align*}

\bigskip

{\bf 7. (Proof of estimate \ref{Fbound1})} With $s, j, \delta, \gamma$ as
in (\ref{isj}), definition of $v_i$ in step 4 yields:
\begin{equation*}
\begin{split}
\|\nabla^{(m+1)}(v_i-v_{i-1})\|_0 & \leq C
\sum_{p+q=m+1}\lambda_i^{p-1}\|\nabla^{(q)}a_s\|_0 \\ & \leq 
C\lambda_i^m\sum_{q=0}^{m+1}\frac{1}{\lambda_i^q}\frac{\lambda_{2s}^q}{(\lambda
  l)^{s/2}}\big(\lambda_0\lambda_2\ldots\lambda_{2s}\big)^{\gamma/2}\big(\|\mathcal{D}\|_0^{1/2}
+ lM\big),
\end{split}
\end{equation*}
where we used the induction assumption \ref{Fbound3}.
The first bound in \ref{Fbound1} then follows, because
$\lambda_{2s}\leq \lambda_i$, due to $2s<i$. For bounding the
$w$-increment we write, recalling (\ref{w_simp}):
\begin{equation}\label{ww}
\begin{split}
& \|\nabla^{(m+1)}(w_i-w_{i-1})\|_0 \leq
C\sum_{p+q+t=m+1}\lambda_i^{p-1}\|\nabla^{(q)}a_s\|_0 \|\nabla^{(t+1)}v_{jk}\|_0 
\\ & \qquad \qquad + C\sum_{p+q+t=m+1}\Big(\lambda_i^{p-2}\|\nabla^{(q)}a_s\|_0 \|\nabla^{(t+1)}a_s\|_0 
+ \lambda_i^{p-1}\|\nabla^{(q)}a_s\|_0 \|\nabla^{(t)}a_s\|_0 \Big)
\end{split}
\end{equation}
We split the first term in the right hand side above, according to
whether $t=0$ or $t\geq 1$:
\begin{equation}\label{waww}
\begin{split}
&\sum_{p+q+t=m+1}\lambda_i^{p-1}\|\nabla^{(q)}a_s\|_0
\|\nabla^{(t+1)}v_{jk}\|_0 \\ & \qquad = \sum_{p+q=m+1}\lambda_i^{p-1}\|\nabla^{(q)}a_s\|_0 \|\nabla v_{jk}\|_0 
+ \sum_{p+q+t=m}\lambda_i^{p-1}\|\nabla^{(q)}a_s\|_0 \|\nabla^{(t+2)}v_{jk}\|_0 
\\ &\qquad \leq C\lambda_i^m\sum_{q=0}^{m+1}\frac{1}{\lambda_i^q}\frac{\lambda_{2s}^q}{(\lambda
l)^{s/2}}\big(\lambda_0\lambda_2\ldots
\lambda_{2s}\big)^{\gamma/2}\big(\|\mathcal{D}\|_0^{1/2}+lM\big)\|\nabla v_{jk}\|_0
\\ & \qquad\qquad + C\lambda_i^m\sum_{q+t=0\ldots m}\frac{1}{\lambda_i^{q+t+1}}\frac{\lambda_{2s}^q}{(\lambda
l)^{s/2}}\big(\lambda_0\lambda_2\ldots
\lambda_{2s}\big)^{\gamma/2}\big(\|\mathcal{D}\|_0^{1/2}+lM\big)\|\nabla^{(t+2)} v_{jk}\|_0
\\ & \qquad \leq C \frac{\lambda_i^m}{(\lambda l)^{s/2}} \big(\lambda_0\lambda_2\ldots
\lambda_{2s}\big)^{\gamma/2}\big(\|\mathcal{D}\|_0^{1/2}+lM\big)
\Big(\|\nabla v_{jk}\|_0 + \sum_{t=0}^m\frac{1}{\lambda_i^{t+1}}\|\nabla^{(t+2)}v_{jk}\|_0\Big)
\end{split}
\end{equation}
For every $t= 0\ldots m+1$, the inductive assumption \ref{Fbound1} gives:
\begin{equation}\label{wawa}
\begin{split}
\|\nabla^{(t+1)}v_{jk}\|_0 & \leq \|\nabla^{(t+1)}v_{0}\|_0 + \sum_{q=1}^{jk}\|\nabla^{(t+1)}(v_q-v_{q-1})\|_0
\\ & \leq \|\nabla^{(t+1)}v_{0}\|_0 +
C \sum_{q=1}^{jk}\frac{\lambda_q^t}{(\lambda
  l)^{s(q)/2}}\big(\lambda_0\lambda_2\ldots\lambda_{2s(q)}\big)^{\gamma/2} \big(\|\mathcal{D}\|_0^{1/2}+lM\big)
\end{split}
\end{equation}
Hence, for the case $t=0$ in (\ref{waww}), in virtue of \ref{pr_stima1} and since $jk<i$, we get directly:
\begin{equation*}
\begin{split}
\|\nabla v_{jk}\|_0 & \leq ClM + \|\nabla v\|_0 + 
C \sum_{q=1}^{jk}\frac{1}{(\lambda
  l)^{s(q)/2}}\big(\lambda_0\lambda_2\ldots\lambda_{2s(q)}\big)^{\gamma/2}
\big(\|\mathcal{D}\|_0^{1/2}+lM\big)\\ &
\leq ClM + \|\nabla v\|_0 + C \big(\lambda_0\lambda_2\ldots\lambda_{2s(jk)}\big)^{\gamma/2}
\big(\|\mathcal{D}\|_0^{1/2}+lM\big) \\ & 
\leq  C \big(\lambda_0\lambda_2\ldots\lambda_{2s(i)}\big)^{\gamma/2}
\big(\|\mathcal{D}\|_0^{1/2}+lM +\|\nabla v\|_0\big).
\end{split}
\end{equation*}
The same bound for $t=1\ldots m+1$, in view of \ref{pr_stima3}, implies that:
\begin{equation*}
\begin{split}
 \sum_{t=0}^m\frac{1}{\lambda_i^{t+1}}\|\nabla^{(t+2)}v_{jk}\|_0 &\leq 
C\sum_{t=0}^m \bigg(\frac{lM}{(\lambda_il)^{t+1}} + \sum_{q=1}^{jk}\frac{1}{(\lambda
  l)^{s(q)/2}}\big(\lambda_0\lambda_2\ldots\lambda_{2s(q)}\big)^{\gamma/2}
\big(\|\mathcal{D}\|_0^{1/2}+lM\big)\bigg) \\ & \leq C
\big(\lambda_0\lambda_2\ldots\lambda_{2s(i)}\big)^{\gamma/2} 
\big(\|\mathcal{D}\|_0^{1/2}+lM\big).
\end{split}
\end{equation*}
Thus, by (\ref{waww}) we see that the first term in the right hand
side of (\ref{ww}) is bounded by:
$$C \frac{\lambda_i^m}{(\lambda l)^{s/2}} \big(\lambda_0\lambda_2\ldots
\lambda_{2s}\big)^{\gamma/2}\big(\|\mathcal{D}\|_0^{1/2}+lM\big)
\big(\|\mathcal{D}\|_0^{1/2}+lM +\|\nabla v\|_0\big)$$
On the other hand, the second term in the right hand side of
(\ref{ww}) is likewise bounded by:
\begin{equation*}
\begin{split}
& C\lambda_m^i\sum_{q+t=0\ldots m+1}\Big(\frac{\lambda_{2s}^{q+t+1}}{\lambda_i^{q+t+1}}
+ \frac{\lambda_{2s}^{q+t}}{\lambda_i^{q+t}}\Big)\frac{1}{(\lambda
  l)^s}\big(\lambda_0\lambda_2\ldots\lambda_{2s}\big)^\gamma \big(\|\mathcal{D}\|_0+(lM)^2\big)
\\ & \leq C \frac{\lambda_i^m}{(\lambda l)^s} \big(\lambda_0\lambda_2\ldots\lambda_{2s}\big)^\gamma
\big(\|\mathcal{D}\|_0+(lM)^2\big), 
\end{split}
\end{equation*}
by \ref{Fbound3} and since $\lambda_{2s}\leq \lambda_i$. This completes the proof of the second
estimate in \ref{Fbound1}.

\medskip

{\bf 8. (Proof of estimate \ref{Fbound2})} Let $i\in (kj, k(j+1)]\cap
(2(s-1), 2s]$ with $j=0\ldots J-1$, $s=1\ldots S$, and denote  $\delta = i-2(s-1)$.
From (\ref{prep_defi_s}) we see that for all $m\geq 0$:
\begin{equation}\label{Vipom}
\begin{split}
& \big\|\nabla^{(m)} \big(V_i -a_{s-1}^2e_\delta\otimes e_\delta\big)\big\|_0 \leq 
C \sum_{p+q+t=m} \lambda_i^{p-1}\|\nabla^{(q)}a_{s-1}\|_0\|\nabla^{(t+2)}v_{jk}\|_0 
\\ & \quad + C \sum_{p+q+t=m}
\lambda_i^{p-2}\Big(\|\nabla^{(q+1)}a_{s-1}\|_0 \|\nabla^{(t+1)}a_{s-1}\|_0 
+ \|\nabla^{(q)}a_{s-1}\|_0\|\nabla^{(t+2)}a_{s-1}\|_0 \Big).
\end{split}
\end{equation}
By \ref{Fbound3}, (\ref{wawa}), \ref{pr_stima3} and the fact that
$\lambda_i\leq \lambda_{2s}$ we get:
\begin{equation*}
\begin{split}
& \sum_{p+q+t=m}  \lambda_i^{p-1}\|\nabla^{(q)}a_{s-1}\|_0\|\nabla^{(t+2)}v_{jk}\|_0 
\\ & \quad \leq C\lambda_i^m \sum_{q+t=0\ldots m}\frac{1}{\lambda_i^{q+t+1}}
\frac{\lambda_{2(s-1)}^q}{(\lambda l)^{(s-1)/2}}
\big(\lambda_0\lambda_2\ldots\lambda_{2(s-1)}\big)^{\gamma/2}\big(\|\mathcal{D}\|_0^{1/2}+ 
lM\big) \times \\ & \quad \qquad \qquad \qquad \quad \times \bigg( \frac{lM}{l^{t+1}}
+\sum_{r=1}^{jk}\frac{\lambda_r^{t+1}}{(\lambda l)^{s(r)/2}}
\big(\lambda_0\lambda_2\ldots\lambda_{2(s-1)}\big)^{\gamma/2}
\big(\|\mathcal{D}\|_0^{1/2}+lM\big)\bigg)
\\ & \quad \leq C \frac{\lambda_i^m}{(\lambda l)^{(s-1)/2}}
\big(\lambda_0\lambda_2\ldots\lambda_{2(s-1)}\big)^{\gamma}
\big(\|\mathcal{D}\|_0+(lM)^2\big) \sum_{t=0}^m\bigg(\frac{1}{(\lambda_il)^{t+1}} + 
\sum_{r=1}^{jk}\frac{\lambda_r^{t+1}}{\lambda_i^{t+1}(\lambda l)^{s(r)/2}}\bigg)
\\ & \quad \leq C \lambda_{2s}^m
\big(\lambda_0\lambda_2\ldots\lambda_{2(s-1)}\big)^{\gamma}
\big(\|\mathcal{D}\|_0+(lM)^2\big)
\bigg(\frac{1}{(\lambda_il) (\lambda l)^{(s-1)/2}} + 
\sum_{r=1}^{jk}\frac{\lambda_r}{\lambda_i (\lambda l)^{s(r)/2} (\lambda l)^{(s-1)/2}}\bigg).
\end{split}
\end{equation*}
Recalling (\ref{count_lam}) and noting that $j(jk)\leq j(i)-1$, we check the following:
\begin{equation*}
\begin{split}
&\frac{1}{(\lambda_il) (\lambda l)^{(s-1)/2}} = \frac{1}{(\lambda l)^{(s-1)/2+1+j(i)}
(\lambda l)^{(s-1)/2}}\leq \frac{1}{(\lambda l)^s},
\\ & \sum_{r=1}^{jk}\frac{\lambda_r}{\lambda_i (\lambda l)^{s(r)/2} (\lambda l)^{(s-1)/2}}
= \sum_{r=1}^{jk}\frac{(\lambda l)^{1+j(r)+s(r)/2}}{(\lambda
  l)^{1+j(i)+s(i)/2}(\lambda l)^{s(r)/2}(\lambda l)^{(s-1)/2}} \\ &
\qquad\qquad\qquad \qquad\qquad \quad 
\leq C \frac{(\lambda l)^{j(jk)}}{(\lambda l)^{j(i)}(\lambda l)^{s-1}} \leq \frac{C}{(\lambda l)^s}.
\end{split}
\end{equation*}
Inserting the above into the previous estimate, we see that
the first term in the right hand side of (\ref{Vipom}) is bounded by: 
$$C \frac{\lambda_{2s}^m}{(\lambda l)^s}
\big(\lambda_0\lambda_2\ldots\lambda_{2(s-1)}\big)^{\gamma} 
\big(\|\mathcal{D}\|_0+(lM)^2\big).$$
On the other hand, for the second term in the right hand side of (\ref{Vipom}) we have:
\begin{equation*}
\begin{split}
& C\lambda_i^m\sum_{q+t=0\ldots m} 
\frac{1}{\lambda_i^{q+t+2}}\frac{\lambda_{2(s-1)}^{q+t+2}}{(\lambda l)^{s-1}}
\big(\lambda_0\lambda_2\ldots\lambda_{2(s-1)}\big)^\gamma\big(\|\mathcal{D}\|_0+ (lM)^2\big)
\\ & \qquad \leq C \frac{\lambda_{2s}^m}{(\lambda l)^s}
\big(\lambda_0\lambda_2\ldots\lambda_{2(s-1)}\big)^\gamma\big(\|\mathcal{D}\|_0+
(lM)^2\big),
\end{split}
\end{equation*}
because:
$$\frac{\lambda_{2(s-1)}}{\lambda_i} \leq
\frac{\lambda_{2(s-1)}}{\lambda_{2(s-1)+1}} = \frac{1}{(\lambda l)^{1/2}}.$$
This ends the proof of \ref{Fbound2} and the proof of our inductive estimates.

\medskip

{\bf 9. (End of proof)} 
We now show that \ref{Fbound1}-\ref{Fbound3} imply \ref{Abound1}-\ref{Abound3} .
First, taking $m=-1,0$ and in view of
(\ref{vw_fin}), \ref{pr_stima1} we conclude a preliminary version of \ref{Abound1}:
\begin{equation*}
\begin{split}
& \|\tilde v- v\|_1\leq \|v_0-v\|_1 + \sum_{i=1}^N \|v_i - v_{i-1}\|_1\leq C \big(\lambda_0\lambda_2\ldots
\lambda_N\big)^{\gamma/2}\big(\|\mathcal{D}\|_0^{1/2}+lM\big),\\
& \|\tilde w - w\|_1\leq \|w_0-w\|_1 + \sum_{i=1}^N \|w_i -
w_{i-1}\|_1 + \sum_{s=0}^{S-1}\|\Psi_s\|_1 \\ & \qquad\qquad \leq  C 
\big(\lambda_0\lambda_2\ldots
\lambda_N\big)^{\gamma}\big(\|\mathcal{D}\|_0^{1/2}+lM\big)\big(1+\|\mathcal{D}\|_0^{1/2}
+lM + \|\nabla v\|_0\big).
\end{split}
\end{equation*}
Taking $m=1$, by \ref{pr_stima3} and since
$1+j(N) = J$, we get a version of the first bound in \ref{Abound2}:
\begin{equation*}
\begin{split}
\|\nabla^2\tilde v\|_0 & = \|\nabla^2 v_N\|_0 \leq \|\nabla^2v_0\|_0 + \sum_{i=1}^N \|\nabla^2(v_i -
v_{i-1})\|_0 \\ & \leq C\Big( \frac{1}{l} +
\sum_{i=1}^N\frac{\lambda_i}{(\lambda l)^{s(i)/2}} \Big)
\big(\lambda_0\lambda_2\ldots \lambda_N\big)^{\gamma/2} \big(\|\mathcal{D}\|_0^{1/2}+lM\big)
\\ & \leq C\Big(\frac{1}{l} + \frac{(\lambda l)^{1+j(N)}}{l}\Big)
\big(\lambda_0\lambda_2\ldots \lambda_N\big)^{\gamma/2} \big(\|\mathcal{D}\|_0^{1/2}+lM\big)
\\ & = C \frac{(\lambda l)^J}{l} 
\big(\lambda_0\lambda_2\ldots \lambda_N\big)^{\gamma/2} \big(\|\mathcal{D}\|_0^{1/2}+lM\big).
\end{split}
\end{equation*}
Similarly, we also get the second bound, recalling again (\ref{count_lam}):
\begin{equation*}
\begin{split}
\|\nabla^2\tilde w\|_0& = \|\nabla^2w_N-\sum_{s=0}^{S-1}\nabla^2\Psi_s\|_0\leq \|\nabla^2w_0\|_0
+ \sum_{i=1}^N \|\nabla^2(w_i - w_{i-1})\|_0 + \sum_{s=0}^{S-1}\|\Psi_s\|_2 \\ & \leq C\Big( \frac{1}{l} +
\sum_{i=1}^N\frac{\lambda_i}{(\lambda l)^{s(i)/2}}
+\sum_{s=0}^{S-1}\frac{\lambda_{2s}}{(\lambda l)^s}\Big)
\big(\lambda_0\lambda_2\ldots \lambda_N\big)^{\gamma} \times \\ &
\qquad \times \big(\|\mathcal{D}\|_0^{1/2}+lM\big)
\big(1+ \|\mathcal{D}\|_0^{1/2}+lM + \|\nabla v\|_0\big) \\ &
\leq C \frac{(\lambda l)^J}{l} \big(\lambda_0\lambda_2\ldots
\lambda_N\big)^{\gamma} \big(\|\mathcal{D}\|_0^{1/2}+lM\big) 
\big(1+ \|\mathcal{D}\|_0^{1/2}+lM + \|\nabla v\|_0\big).
\end{split}
\end{equation*}
Finally, (\ref{pomag}), \ref{pr_stima2}, and \ref{Fbound2} applied with $m=0$ yield a
version of \ref{Abound3}:
\begin{equation*}
\begin{split}
\|\tilde{\mathcal{D}}\|_0 = & ~ \|(A-A_0) -\mathcal{D}_S\|_0 \leq \|A-A_0
\|_0 + \|\mathcal{D}_S\|_0 \\ & \leq  C\Big(l^\beta\|A\|_{0,\beta} +
\frac{1}{(\lambda l)^S} \big(\lambda_0\lambda_2\ldots
\lambda_{2(S-1)}\big)^\gamma \big(\|\mathcal{D}\|_0+ (lM)^2\big)\Big)
\end{split}
\end{equation*}
We conclude the final estimates by a straightforward calculation in:
$$\lambda_0\lambda_2\ldots \lambda_N = \frac{\displaystyle{\prod_{p=1}^{N/2}}(\lambda
l)^{1+j(2p)+(p-1)/2}}{l^{N/2+1}} = \left\{\begin{array}{ll}
\displaystyle{\frac{(\lambda l)^{(k^2+6k)/16}}{l^{k/2+1}} } & \mbox{
  for $k$ even},\vspace{3mm} \\
\displaystyle{\frac{(\lambda l)^{(k^2+5k+2)/4}}{l^{k+1}}} & \mbox{ for $k$ odd},
\end{array}\right. $$
which implies that:
$$\lambda_0\lambda_2\ldots \lambda_N \leq \frac{(\lambda
  l)^{(k^2+5k+2)/4}}{l^{k+1}}\leq \lambda^{(k^2+5k+2)/4}.$$
Thus, we achieve \ref{Abound1}-\ref{Abound3} with the auxiliary
exponent $\frac{k^2+5k+2}{4}\gamma$, rather than $\gamma$. These yield
the claimed bounds as well, by a simple re-parametrisation. The proof is done.
\endproof

\section{The Nash-Kuiper scheme in $\mathcal{C}^{1,\alpha}$: a proof
  of Theorem \ref{th_NashKuiHol}}\label{sec4} 

Before giving the proof, we note that  taking $S,J$ as in Theorem
\ref{thm_stage}, for which (\ref{lcm_defe}) implies:
$$\frac{S}{S+2J} = \frac{1}{1+4/k},$$
Theorem \ref{th_NashKuiHol} automatically yields the following result,
particular to dimension $d=2$:

\begin{corollary}\label{th_NKH}
Given an open, bounded, smooth domain $\omega\subset\R^2$, there
exists $l_0\in (0,1)$ such that the following holds for every $l\in (0,l_0)$.
For every $v\in\mathcal{C}^2(\bar\omega + \bar B_{2l}(0),\R^k)$,
$w\in\mathcal{C}^2(\bar\omega +\bar B_{2l}(0),\R^2)$
and $A\in\mathcal{C}^{0,\beta}(\bar\omega +\bar B_{2l}(0), \R^{2\times 2}_\sym)$, such that:
$$\mathcal{D}=A-\big(\frac{1}{2}(\nabla v)^T\nabla v + \sym\nabla
w\big) \quad\mbox{ satisfies } \quad 0<\|\mathcal{D}\|_0\leq 1,$$
and for every $\alpha$ in the range:
\begin{equation}\label{rangeAl_2}
0< \alpha <\min\Big\{\frac{\beta}{2},\frac{1}{1+4/k}\Big\},
\end{equation}
there exist $\tilde v\in\mathcal{C}^{1,\alpha}(\bar\omega,\R^k)$ and
$\tilde w\in\mathcal{C}^{1,\alpha}(\bar\omega,\R^2)$ with the following properties:
\begin{align*}
& \|\tilde v - v\|_1\leq C(1+\|\nabla v\|_0)^2\|\mathcal{D}\|_0^{1/4}, \quad \|\tilde w -
w\|_1\leq C (1+\|\nabla v\|_0)^3\|\mathcal{D}\|_0^{1/4},
\tag*{(\theequation)$_1$}\refstepcounter{equation} \label{Hbound1_2}\vspace{1mm}\\
& A-\big(\frac{1}{2}(\nabla \tilde v)^T\nabla \tilde v + \sym\nabla
\tilde w\big) =0 \quad\mbox{ in }\; \bar\omega. \tag*{(\theequation)$_2$}\label{Hbound2_2} 
\end{align*}
%\end{equation}
The norms in
the left hand side of \ref{Hbound1_2} are taken on $\bar\omega$, and in the right hand
side on $\bar\omega+ \bar B_{2l}(0)$. The constants $C$ depend only
on $\omega, k, A$ and $\alpha$. 
\end{corollary}

\medskip

\noindent The remaining part of this section will be devoted to:

\bigskip

\noindent {\bf Proof of Theorem \ref{th_NashKuiHol}}

\smallskip

{\bf 1.} We set $v_0=v, w_0=w, \mathcal{D}_0=\mathcal{D}$.  
Then, for each $i\geq 1$ we will define:
$$ v_i\in\mathcal{C}^2(\bar\omega + \bar B_{l_i}(0),\R^k),  \quad
w_i\in \mathcal{C}^2(\bar\omega + \bar B_{l_i}(0),\R^d),\quad
\mathcal{D}_i=A-\big(\frac{1}{2} (\nabla v_i)^T\nabla v_i + \sym\nabla w_i\big),$$
by applying Theorem \ref{thm_stage} to $v_{i-1}$, $w_{i-1}$, $A$,
with specific parameters $\gamma, l_{i-1}, \lambda_{i-1}, M_{i-1}$.
To this end, we will define $\gamma\in (0,1)$, $\{l_i\}_{i=1}^\infty$, 
$\big\{\lambda_i, M_i\}_{i=0}^\infty $ satisfying, as below, the bounds
for all $i\geq 0$ and convergences as $i\to\infty$. We may also
decrease $l_0$ if needed. Namely, we will require:
\begin{equation}\label{propl}
\begin{split}
& l_{i+1}\leq \frac{l_{i}}{2},\quad  l_i\lambda_i>1,
\quad M_i\geq \max\{\|v_i\|_2,\|w_i\|_2,1\},\quad M_i\nearrow \infty,
%\quad \sum_{i=0}^\infty\lambda_i^\gamma l_iM_i\leq 2\|\mathcal{D}_0\|_0^{1/2},
\\ & \|\mathcal{D}_i\|_0\leq (l_iM_i)^2, \quad l_iM_i\to 0.
\end{split}
\end{equation}
From \ref{Abound1}-\ref{Abound3} we then get for all $i\geq 0$:
\begin{equation}\label{Abd1}
\begin{split}
& \|v_{i+1} - v_i\|_1\leq C\lambda_i^{\gamma} l_iM_i, \qquad
\|w_{i+1} - w_i\|_1\leq C\lambda_i^{\gamma} l_iM_i \big(1+l_iM_i+\|\nabla v_i\|_0\big), \\
& \| v_{i+1}\|_2\leq C (\lambda_i l_i)^J\lambda_i^{\gamma} M_i,\qquad
\|w_{i+1}\|_2\leq C (\lambda_i l_i)^J\lambda_i^{\gamma}M_i \big(1+l_iM_i+\|\nabla v_i\|_0\big), \\ 
& \|{\mathcal{D}_{i+1}}\|_0\leq C\Big(l_i^\beta{\|A\|_{0,\beta}}
+\frac{1}{(\lambda_i l_i)^S} \lambda_i^{\gamma} (l_iM_i)^2\Big).
\end{split}
\end{equation}
The above bound on $\|v_{i+1}\|_2$ follows by:
\begin{equation*}
\begin{split}
\| v_{i+1}\|_2&\leq \|\nabla^2 v_{i+1}\|_0+\|v_{i+1}-v_i\|_1 +
\|v_i\|_1\\ & \leq C (\lambda_i l_i)^J\lambda^{\gamma} M_i +
C\lambda_i^{\gamma} l_iM_i+M_i\leq C (\lambda_i l_i)^J\lambda^{\gamma}M_i,
\end{split}
\end{equation*}
in view of \ref{Abound2}, \ref{Abound1} and (\ref{propl}). The bound on  $\| w_{i+1}\|_2$ is obtained similarly.
We also recall our convention that all constants denoted by $C$ are bigger
than $1$ and may change from line to line, but depend only on
$\omega$, $k$, $\gamma$ (and $S,J\geq 1$ in the present proof). 

\medskip

{\bf 2.} To show the validity of (\ref{propl}), we make the following ansatz:
\begin{equation}\label{lal}
\lambda_i = \frac{b}{l_i^a} \quad \mbox{ with some }\; a\in (1,2)
\;\mbox{ and } \;b>1.
\end{equation}
Anticipating the details of the proof, it is convenient to keep in
mind that we assign: sufficiently large $b$, $l_0$ small and
$(a-1)$ small, and $\gamma$ small. The requirements in (\ref{propl}) are then implied by:
\begin{equation}\label{propl2}
\begin{split}
& l_{i+1}\leq \frac{l_{i}}{2}, \quad
M_i\geq \max\{\|v_i\|_2,\|w_i\|_2,1\},\quad M_i\nearrow \infty,
\\ &  b^\gamma\sum_{i=0}^\infty l_i^{1-a\gamma} M_i\leq
C\frac{b^{(S+2J)\gamma}}{l_0^{2a\gamma}} \big(1+\|\nabla v_0\|_0\big)\|\mathcal{D}_0\|_0^{1/2}, 
\\ & \|\mathcal{D}_i\|_0\leq (l_iM_i)^2, 
\end{split}
\end{equation}
whereas the bounds in (\ref{Abd1}) may be rewritten as:
\begin{equation}\label{Abd2}
\begin{split}
& \|v_{i+1} - v_i\|_1\leq C b^\gamma l_i^{1-a\gamma} M_i, \qquad
\|w_{i+1} - w_i\|_1\leq Cb^\gamma l_i^{1-a\gamma} \frac{b^{(S+2J)\gamma}}{l_0^{2a\gamma}} M_i \big(1+\|\nabla
v_0\|_0\big), \\
& \|v_{i+1}\|_2\leq C \frac{b^{J+\gamma}}{l_i^{(a-1)J + a\gamma}} M_i,\qquad
\|w_{i+1}\|_2\leq C \frac{b^{J+\gamma}}{l_i^{(a-1)J +a\gamma}} 
\frac{b^{(S+2J)\gamma}}{l_0^{2a\gamma}} M_i \big(1+\|\nabla v_0\|_0\big),  \\
& \|{\mathcal{D}_{i+1}}\|_0\leq C\Big(l_i^\beta{\|A\|_{0,\beta}}
+\frac{l_i^{2+(a-1)S-a\gamma}}{b^{S-\gamma}} M_i^2\Big)
\end{split}
\end{equation}
The middle two bounds above imply that:
$$\max\{\|v_{i+1}\|_2,\|w_{i+1}\|_2,1\} \leq C 
\frac{b^{J+(S+2J+1)\gamma}}{l_i^{(a-1)J+3\gamma a}} M_i \big(1+\|\nabla v_0\|_0\big).$$
Consequently, and splitting the last bound in (\ref{Abd2}) between the
two terms in its right hand side, we
see that the requirements in (\ref{propl2}) are implied by the satisfaction of:
\begin{equation}\label{propl3}
\begin{split}
& l_{i+1}\leq \frac{l_{i}}{2}, \quad \|\mathcal{D}_0\|_0 = (l_0M_0)^2,
\\ &  b^\gamma\sum_{i=0}^\infty l_i^{1-a\gamma} M_i\leq
C\frac{b^{(S+2J)\gamma}}{l_0^{2a\gamma}} \big(1+\|\nabla v_0\|_0\big)\|\mathcal{D}_0\|_0^{1/2}, 
\\ & \Big(\frac{M_{i+1}}{M_i}\Big)^2\geq \max\Big\{
\frac{2C}{b^{S-\gamma}}\frac{l_i^{2+(a-1)S-a\gamma}}{l_{i+1}^2},
\frac{C b^{2J+2(S+2J+1)\gamma}}{l_i^{2(a-1)J+6\gamma a}}\Big\}\cdot \big(1+\|\nabla v_0\|_0\big)^2,\\ 
& M_{i+1}^2\geq \frac{2C l_i^\beta}{l_{i+1}^2}\|A\|_{0,\beta}.
\end{split}
\end{equation}
In the right hand side of second line estimate above, its first term prevails provided that:
$$l_{i+1}^2 \leq \frac{l_i^{2+(a-1)(S+2J)5+a\gamma}}{C b^{S+2J+(2S+4J+1)\gamma} }$$
Consequently, we define:
\begin{equation}\label{lal2}
\begin{split}
 l_{i}= B^{\frac{q^i-1}{q-1}} l_0^{q^i}
\quad & \mbox{ where }\quad \frac{1}{B}= Cb^{\frac{S}{2} +J+(S+2J+\frac{1}{2})\gamma}
\\ & \mbox{ and } \quad q = 1+(a-1) \big(\frac{S}{2}+J\big) +\frac{5a\gamma}{2}
\end{split}
\end{equation}
and note that the estimates in (\ref{propl3}) are then guaranteed by:
\begin{align*}
& b^\gamma\sum_{i=0}^\infty B^{\frac{(q^i-1)(1-a\gamma)}{q-1}}l_0^{q^i(1-a\gamma)} M_i\leq
C\frac{b^{(S+2J)\gamma}}{l_0^{2a\gamma}} \big(1+\|\nabla v_0\|_0\big)\|\mathcal{D}_0\|_0^{1/2}, 
\tag*{(\theequation)$_1$}\refstepcounter{equation} \label{propl4_1}\medskip 
\\ & \frac{M_{i+1}^2}{M_i^2}\geq 
\frac{2C}{b^{S-\gamma}}\frac{1}{B^{\frac{(a-1)S-a\gamma}{q-1}}}
\frac{1}{\big(B^{\frac{1}{q-1}}l_0\big)^{q^i(2J(a-1)+ga\gamma)}}\big(1+\|\nabla v_0\|_0\big),
\tag*{(\theequation)$_2$}\label{propl4_2}\medskip 
\\ & M_{i+1}^2\geq 2 C
\|A\|_{0,\beta}\frac{B^{\frac{2-\beta}{q-1}}}{\big(B^{\frac{1}{q-1}}l_0\big)^{q^i(2q-\beta)}}.
\tag*{(\theequation)$_3$}\label{propl4_3}\medskip 
\end{align*}
We will assume the initial normalisation:
$$\|\mathcal{D}_0\|_0 = (l_0M_0)^2,$$
and show \ref{propl4_1}-\ref{propl4_3} for all $i\geq 0$, 
by separating our construction into two cases below.

\medskip

{\bf 3.} We start by observing that condition \ref{propl4_2} holds, if we set:
\begin{equation}\label{def2}
M_{i+1}^2 = M_0^2 \Big(\frac{2C}{b^{S-\gamma}} \frac{(1+\|\nabla v_0\|_0)^2}{B^{\frac{S(a-1)-a\gamma}{q-1}}}\Big)^{i+1}
\big(B^{\frac{1}{q-1}}l_0\big)^{\frac{2J(a-1)+6a\gamma}{q-1}} 
\frac{1}{\big(B^{\frac{1}{q-1}}l_0\big)^{q^{i+1} \frac{2J(a-1)+6a\gamma}{q-1}}},
\end{equation}
for all $i\geq 0$. On the other hand, \ref{propl4_3} follows directly, by assigning:
\begin{equation}\label{def1}
M_{i+1}^2 = \big(M_0 (1+\|\nabla v_0\|_0)\big)^{2(i+1)}l_0^{2-\beta}
\frac{B^{\frac{2-\beta}{q-1}}}{\big(B^{\frac{1}{q-1}}l_0\big)^{q^{i} (2q -\beta)}},
\end{equation}
and taking $M_0^2l_0^{2-\beta}\geq 8C \|A\|_{0,\beta}$, which is
guaranteed by assigning $l_0$ small enough to have: 
\begin{equation}\label{reqi3}
2C \|A\|_{0,\beta} l_0^\beta\leq \|\mathcal{D}_0\|_0.
\end{equation}

\medskip

\noindent We now choose the larger one of definitions (\ref{def2}), (\ref{def1}), asymptotically as
$i\to \infty$, which in view of $l_0, B<1$ reduces to choosing the
larger exponent in the power of $\frac{1}{B^{\frac{1}{q-1}}l_0}$. There holds:
\begin{equation*}
\begin{split}
\displaystyle \frac{\beta}{2} >\frac{S}{S+2J} 
\implies 2q-\beta & <2\frac{J+\frac{3a}{a-1}\gamma}{\big(\frac{S}{2}+J\big)
  +\frac{5a}{2(a-1)}\gamma }\\ & =\frac{2}{q-1} \big(J(a-1)+3a\gamma\big)<
\frac{2q}{q-1} \big(J(a-1)+3a\gamma\big),
\end{split}
\end{equation*}
provided that $a-1$ small and $\gamma$ small.
In that case we proceed with (\ref{def2}). Further:
$$\displaystyle \frac{\beta}{2} \leq \frac{S}{S+2J} 
\implies \frac{\beta}{2} < q\frac{S-\frac{a}{a-1}\gamma}{\big(S+2J\big)+ \frac{5a}{a-1}\gamma}
\implies 2q-\beta >\frac{2q}{q-1} \big(J(a-1)+3a\gamma\big),$$
by the same order of assigning $a$ and then a small $\gamma$. In that case we
will adopt (\ref{def1}).

\medskip

{\bf 4.} {\bf (Case $\mathbf{\displaystyle \frac{\pmb{\beta}}{2} >
    \frac{S}{S+2J}}$, definition (\ref{def2}))} Below, we show
that \ref{propl4_1} and \ref{propl4_3} may be achieved by assigning
$b, l_0, a, \gamma$  appropriately. We first consider the bound \ref{propl4_3}, which is implied by
the following estimate:
$$M_0^2 \bigg(\frac{2C}{b^{S-\gamma}} \frac{(1+\|\nabla v_0\|_0)^2}{B^{\frac{S(a-1)-a\gamma}{q-1}}}\bigg)^{i+1}
\frac{\big(B^{\frac{1}{q-1}}l_0\big)^{\frac{2J(a-1)+6a\gamma}{q-1}}}{
\big(B^{\frac{1}{q-1}}l_0\big)^{q^i\big(\frac{2J(a-1)+6a\gamma}{q-1}-(2q-\beta)\big)}}
\geq 2C \|A\|_{0,\beta}B^{\frac{2-\beta}{q-1}}.$$
Multiplying both sides by $l_0^2 \big(B^{\frac{1}{q-1}}l_0\big)^{\beta-2q}$ and recalling 
$M_0^2l_0^2=\|\mathcal{D}_0\|_0$, we equivalently write:
$$ \bigg(\frac{2C}{b^{S-\gamma}} \frac{(1+\|\nabla v_0\|_0)^2}{B^{\frac{S(a-1)-a\gamma}{q-1}}}\bigg)^{i+1}
\frac{1}{\big(B^{\frac{1}{q-1}}l_0\big)^{(q^i-1)\big(\frac{2J(a-1)+6a\gamma}{q-1}-(2q-\beta)\big)}}
\geq \frac{2C \|A\|_{0,\beta}}{\|\mathcal{D}_0\|_0}B^{-2}l_0^{\beta-2(q-1)}.$$
Since $q^i-1\geq (q-1)i$, the above is implied by:
$$ \bigg(\frac{2C}{b^{S-\gamma}} \frac{1}{B^{\frac{S(a-1)-a\gamma}{q-1}}}\bigg)^{i+1}
\bigg(\frac{1}{B^{\frac{2J(a-1)+6a\gamma}{q-1}-(2q-\beta)}}\bigg)^i
\geq \frac{2C \|A\|_{0,\beta}}{\|\mathcal{D}_0\|_0}B^{-2}l_0^{\beta-2(q-1)},$$
and further by:
\begin{equation}\label{mumill} 
\frac{1}{b^{S-\gamma} B^{\frac{S(a-1)-a\gamma}{q-1}}}\frac{B^2}{l_0^{\beta-2(q-1)}}
\bigg(\frac{2C}{b^{S-\gamma}B^{\beta-2(q-1)}}\bigg)^{i}
\geq \frac{2C \|A\|_{0,\beta}}{\|\mathcal{D}_0\|_0}.
\end{equation}
We now observe that the base power quantity in the left hand side above can be written as:
$$\frac{2C}{b^{S-\gamma}B^{\beta-2(q-1)}}= C b^{\big(\frac{S}{2}+J +
  (S+2J+\frac{1}{2})\gamma\big)(\beta-2(q-1))-S+\gamma}\geq 1,$$
as the exponent there is positive, by the first
implication in step 3. Thus, (\ref{mumill}) follows from:
$$\frac{1}{b^{S-\gamma +(S+2J+(2S+4J+1)\gamma)\frac{2J+\frac{6a}{a-1}\gamma}{S+2J
    +\frac{5a}{a-1}\gamma}}}\frac{1}{l_0^{\beta-2(q-1)}}
\geq \frac{2C \|A\|_{0,\beta}}{\|\mathcal{D}_0\|_0},$$
implied, if only $a-1$ and $\gamma$ are small, by: 
\begin{equation}\label{reqi1} 
\frac{1}{b^{S+4J}}\frac{1}{l_0^{\beta-2(q-1)}}\geq\frac{2C \|A\|_{0,\beta}}{\|\mathcal{D}_0\|_0}.
\end{equation}
We will show the validity of this and other requirements in step 6 below.

\medskip

\noindent We now consider the estimate in \ref{propl4_1}, namely:
\begin{equation*}
\begin{split}
b^\gamma \sum_{i=0}^\infty B^{\frac{(q^{i}-1)(1-a\gamma)}{q-1}}& l_0^{q^{i}(1-a\gamma)}M_0
\Big(\frac{2C}{b^{S-\gamma}}\frac{(1+\|\nabla v_0\|_0)^2}{B^{\frac{S(a-1)-a\gamma}{q-1}}}\Big)^{i/2}\frac{1}
{\big(B^{\frac{1}{q-1}}l_0\big)^{(q^i-1)\frac{J(a-1)+3a\gamma}{q-1}}}
\\ & \leq C\frac{b^{(S+2J)\gamma}}{l_0^{2a\gamma}}\big(1+\|\nabla v_0\|_0\big)^2\|\mathcal{D}_0\|_0^{1/2}.
\end{split}
\end{equation*}
In view of $M_0^2l_0^2=\|\mathcal{D}_0\|_0$, the left hand side above can be
equivalently written and estimated:
\begin{equation*}
\begin{split} 
b^\gamma \frac{\|\mathcal{D}_0\|_0^{1/2}}{l_0^{a\gamma}}&\sum_{i=0}^\infty
\Big(\frac{2C}{b^{S-\gamma}}\frac{(1+\|\nabla v_0\|_0)^2}{B^{\frac{S(a-1)-a\gamma}{q-1}}}\Big)^{i/2}\frac{1}
{\big(B^{\frac{1}{q-1}}l_0\big)^{(q^i-1)\big(\frac{J(a-1)+3a\gamma}{q-1}-1+a\gamma\big)}}\\
& \leq {b^\gamma}\frac{\|\mathcal{D}_0\|_0^{1/2}}{l_0^{a\gamma}}\sum_{i=0}^\infty
\big(C (1+\|\nabla v_0\|_0)\big)^i
\big(B^{\frac{1}{q-1}}l_0\big)^{(q^i-1)\big(\frac{S-\frac{a}{a-1}\gamma}{S+2J+\frac{5a}{a-1}\gamma}-a\gamma\big)}
\\ & \leq \|\mathcal{D}_0\|_0^{1/2}\Big(\frac{b}{l_0^a}\Big)^\gamma\sum_{i=0}^\infty
\Big(C(1+\|\nabla v_0\|_0) B^{\frac{S-\frac{a}{a-1}\gamma}{S+2J+\frac{5a}{a-1}\gamma}-a\gamma} \Big)^i
\\ & \leq \|\mathcal{D}_0\|_0^{1/2}\Big(\frac{b}{l_0^a}\Big)^\gamma\sum_{i=0}^\infty
\Bigg(\frac{C(1+\|\nabla v_0\|_0)}{
b^{\big(\frac{S}{2}+J+(S+2J+\frac{1}{2})\gamma)\big)\big(\frac{S-\frac{a}{a-1}\gamma}{S+2J+\frac{5a}{a-1}\gamma}-a\gamma\big)}
}\Bigg)^i,
\end{split}
\end{equation*}
where we again used $q^i-1\geq (q-1)i$ for all $i\geq 0$. The bound in
\ref{propl4_1} will follow, in particular, by assuring that the series
in the right hand side above sums to less than $2$, through:
\begin{equation}\label{reqi2}
2C(1+\|\nabla v_0\|_0)\leq
b^{\big(\frac{S}{2}+J+(S+2J+\frac{1}{2})\gamma\big)\big(\frac{S-\frac{a}{a-1}\gamma}{S+2J+\frac{5a}{a-1}\gamma}-a\gamma\big)} 
\end{equation}

\medskip

{\bf 5.} {\bf (Case $\mathbf{\displaystyle \frac{\pmb{\beta}}{2} \leq
    \frac{S}{S+2J}}$, definition (\ref{def1}))}
In this second case, we show that \ref{propl4_1} and \ref{propl4_2}
are valid with appropriate $b, l_0, a,\gamma$.
We first consider \ref{propl4_2} at $i=0$, which is:
\begin{equation*}
\begin{split}
\frac{1}{B^2} \frac{(1+\|\nabla v_0\|_0)^2}{l_0^{(S+2J)(a-1) +5a\gamma}}
=\frac{1 (1+\|\nabla v_0\|_0)^2}{(B^{\frac{1}{q-1}}l_0)^{2(q-1)}}= \frac{M_1^2}{M_0^2} & \geq 
\frac{2C}{b^{S-\gamma}}\frac{1}{B^{\frac{S(a-1)-a\gamma}{q-1}}}
\frac{(1+\|\nabla v_0\|_0)^2}{\big(B^{\frac{1}{q-1}}l_0\big)^{2J(a-1)+6a\gamma}} \\ & = 
\frac{2C}{b^{S-\gamma}}\frac{1}{B^2} \frac{(1+\|\nabla v_0\|_0)^2}{l_0^{2J(a-1)+6a\gamma}}. 
\end{split}
\end{equation*}
Thus, for the validity of the above requirement, we necessitate:
\begin{equation}\label{reqi4}
2C l_0^{S(a-1)-a\gamma}\leq b^{S-\gamma},
\end{equation}
which is achieved with $b$ large. To complete the analysis of
\ref{propl4_2}, we will show for all $i\geq 0$:
$$ \frac{M_0^2}{\big(B^{\frac{1}{q-1}}l_0\big)^{q^i(q-1)(2q-\beta)}}\geq 
\frac{2C}{b^{S-\gamma}}\frac{1}{B^{\frac{S(a-1)-a\gamma}{q-1}}}
\frac{1}{\big(B^{\frac{1}{q-1}}l_0\big)^{q^i(2J(a-1)+6a\gamma)}},$$
which is equivalent to:
\begin{equation}\label{dars2}
M_0\geq \frac{2C}{b^{S-\gamma}}\frac{1}{B^{\frac{S(a-1)-a\gamma}{q-1}}}
\frac{1}{\big(B^{\frac{1}{q-1}}l_0\big)^{q^i\big(2J(a-1)+6a\gamma-(q-1)(2q-\beta)\big)}}.
\end{equation}
By the second implication in step 3, we note the sign of the exponent:
$$2J(a-1)+6a\gamma-(q-1)(2q-\beta)<\big(2J(a-1)+3a\gamma\big)(1-q)<0$$
Hence, and recalling that $(l_0M_0)^2=\|\mathcal{D}_0\|_0$, it follows
that (\ref{dars2}) is implied by:
$$ \frac{\|\mathcal{D}_0\|_0^{1/2}}{l_0} \geq  \frac{2C}{b^{S-\gamma}}\frac{1}{B^{\frac{S(a-1)-a\gamma}{q-1}}}
= \frac{2C}{b^{S-\gamma-(S+2J+(2S+4J+1)\gamma)
\frac{S-\frac{a}{a-1}\gamma}{S+2J+\frac{5a}{a-1}\gamma}}}.$$
Since the power in the exponent above is positive, we see that it is
enough to assure that:
\begin{equation}\label{reqi5}
2C l_0\leq |\mathcal{D}_0\|_0^{1/2}.
\end{equation}
We will show the validity of this and other requirements in step 7 below.

\medskip

\noindent We now validate the estimate in \ref{propl4_1}, namely:
\begin{equation*}
\begin{split}
b^\gamma\bigg( l_0^{1-a\gamma} M_0 + \sum_{i=0}^\infty &B^{\frac{(q^{i+1}-1)(1-a\gamma)}{q-1}} 
l_0^{q^{i+1}(1-a\gamma)}\Big(M_0(1+\|\nabla v_0\|_0)\Big)^{i+1}l_0^{1-\beta/2}\frac{B^{\frac{1-\beta/2}{q-1}}}
{\big(B^{\frac{1}{q-1}}l_0\big)^{q^i(q-\beta/2)}}\bigg)\\ & \leq C
\frac{b^{(2+2J)\gamma}}{l_0^{2a\gamma}}(1+\|\nabla v_0\|_0)\|\mathcal{D}_0\|_0^{1/2}.
\end{split}
\end{equation*}
The left hand side above may be rewritten and estimated by:
\begin{equation*}
\begin{split}
b^\gamma &\bigg( \frac{\|\mathcal{D}_0\|_0^{1/2}}{l_0^{a\gamma}} +
\sum_{i=0}^\infty
\big(B^{\frac{1}{q-1}}l_0\big)^{q^i(\beta/2-aq\gamma)}\frac{M_0^{i+1}(1+\|\nabla v_0\|_0)^{i+1}l_0^{1-\beta/2}}
{B^{\frac{\beta/2-a\gamma}{q-1}}}\bigg) \\ & \leq 
b^\gamma \bigg( \frac{\|\mathcal{D}_0\|_0^{1/2}}{l_0^{a\gamma}} +
M_0 (1+\|\nabla v_0\|_0)\big(B^{\frac{1}{q-1}}l_0\big)^{\beta/2-aq\gamma}
\frac{l_0^{1-\beta/2}}{B^{\frac{\beta/2-aq\gamma}{q-1}}}\times \\ &
\qquad\qquad\qquad\qquad
\times \sum_{i=0}^\infty \Big(\big(B^{\frac{1}{q-1}}l_0\big)^{(q-1)(\beta/2-a\gamma)}M_0 (1+\|\nabla v_0\|_0)\Big)^i\bigg) 
\\ & \leq \frac{b^\gamma\|\mathcal{D}_0\|_0^{1/2}}{l_0^{aq\gamma}} \bigg(
1 + \frac{1+\|\nabla v_0\|_0}{B^{a\gamma}}\sum_{i=0}^\infty
\Big(B^{\beta/2-aq\gamma}\frac{(1+\|\nabla v_0\|_0)\|\mathcal{D}_0\|_0^{1/2}}{l_0}\Big)^i\bigg)
\\ & \leq C\frac{b^{\gamma + (\frac{S}{2}+J+(S+2J+\frac{1}{2})\gamma)a\gamma}}
{l_0^{aq\gamma}}(1+\|\nabla v_0\|_0)\|\mathcal{D}_0\|_0^{1/2}
\bigg( 1 + \sum_{i=0}^\infty\Big(\frac{1+\|\nabla
  v_0\|_0}{l_0 b^{\big(\frac{S}{2}+J+ (S+2J+\frac{1}{2})\gamma\big)(\frac{\beta}{2}-a\gamma)}}\Big)^i\bigg)
\\ & \leq C\frac{b^{\gamma + (\frac{S}{2}+J+(S+2J+\frac{1}{2})\gamma)a\gamma}}
{l_0^{aq\gamma}}(1+\|\nabla v_0\|_0)\|\mathcal{D}_0\|_0^{1/2},
\end{split}
\end{equation*}
where we used the fact that $q^i\geq (q-1)i+1$ for all $i\geq 0$ and
the requirement that the ratio in the geometric series above is less that $\frac{1}{2}$, implied by:
\begin{equation} \label{reqi6}
2 (1+\|\nabla v_0\|_0)\leq l_0 b^{S\beta/{6}},
\end{equation}
which we note automatically implies (\ref{reqi4}).

\medskip

{\bf 6.} {\bf (Case $\mathbf{\displaystyle \frac{\pmb{\beta}}{2} >
    \frac{S}{S+2J}}$, viability of assumptions in step 4 and
  ${\pmb{\mathcal{C}}}^{\mathbf{1}}$ convergence)} 
We now examine (\ref{reqi1}), (\ref{reqi2}). Under the usual
assumption $a-1, \gamma\ll 1$, these are implied by:
\begin{equation}\label{faras}
b^{S/4}\geq C(1+\|\nabla v_0\|_0), \qquad l_0^\beta\leq
\frac{\|\mathcal{D}_0\|_0}{C\|A\|_{0,\beta} b^{S+4J}}.
\end{equation}
Hence we define:
$$ b^{S/4}= C(1+\|\nabla v_0\|_0), \qquad l_0^\beta=
\frac{\|\mathcal{D}_0\|_0}{C\|A\|_{0,\beta} b^{S+4J}}=
\frac{\|\mathcal{D}_0\|_0}{C\|A\|_{0,\beta} (1+\|\nabla v_0\|_0)^{\frac{4}{S}(S+4J)}}.$$
Consequently, the right hand side of the bound in \ref{propl4_1}
becomes, if only $\gamma\ll 1$:
\begin{equation*}
\begin{split}
C\frac{b^{(S+2J)\gamma}}{l_0^{2a\gamma}} &\big(1+\|\nabla v_0\|_0\big)\|\mathcal{D}_0\|_0^{1/2}
\\ & = C\bigg(\frac{(1+\|\nabla v_0\|_0)^{\frac{4}{S}\big(S+2J+\frac{2a}{\beta}(S+4J)\big)} 
\|A\|_{0,\beta}^{2a/\beta}} {\|\mathcal{D}_0\|_0^{2a/\beta}}\bigg)^\gamma  
\big(1+\|\nabla v_0\|_0\big)\|\mathcal{D}_0\|_0^{1/2}
\\ & \leq C \big(1+\|A\|_{0,\beta}\big)^{2a\gamma/\beta}\big(1+\|\nabla v_0\|_0\big)^2
\|\mathcal{D}_0\|_0^{1/2 - 2a\gamma/\beta}
\end{split}
\end{equation*} 
and we likewise observe that:
\begin{equation*}
\begin{split}
C\bigg(\frac{b^{(S+2J)\gamma}}{l_0^{2a\gamma}} &\big(1+\|\nabla v_0\|_0\big)\bigg)^2\|\mathcal{D}_0\|_0^{1/2}
\leq C \big(1+\|A\|_{0,\beta}\big)^{4a\gamma/\beta}\big(1+\|\nabla v_0\|_0\big)^3
\|\mathcal{D}_0\|_0^{1/2 - 4a\gamma/\beta}.
\end{split}
\end{equation*} 
In particular, if $\gamma\leq\frac{\beta}{32}$ so that
$\frac{4a\gamma}{\beta}\leq \frac{1}{4}$, the exponents on the deficits
are greater than $\frac{1}{4}$ and:
$$\big(1+\|A\|_{0,\beta}\big)^{4a\gamma/\beta}\leq  \big(1+\|A\|_{0,\beta}\big)^{1/4},\qquad
\|\mathcal{D}_0\|_0^{1/2 - 4a\gamma/\beta}\leq  \|\mathcal{D}_0\|_0^{1/4}.$$
Recalling (\ref{Abd2}) it now follows that:
\begin{equation}\label{kalb}
\begin{split}
& \sum_{i=0}^{\infty}\|v_{i+1}-v_i\|_1\leq  
C \big(1+\|A\|_{0,\beta}\big)^{1/8}\big(1+\|\nabla v_0\|_0\big)^2 \|\mathcal{D}_0\|_0^{3/8},\\ 
& \sum_{i=0}^{\infty}\|w_{i+1}-w_i\|_1\leq C
\big(1+\|A\|_{0,\beta}\big)^{1/4}\big(1+\|\nabla v_0\|_0\big)^3
\|\mathcal{D}_0\|_0^{1/4}, 
\end{split}
\end{equation}
hence the sequences $\{v_i\}_{i=0}^\infty$, $\{w_i\}_{i=0}^\infty$ converge
in $\mathcal{C}^1(\bar\omega)$ to some limiting fields
$\tilde v\in\mathcal{C}^1(\bar\omega,\R^k)$, $\tilde
w\in\mathcal{C}^1(\bar\omega,\R^d)$ that satisfy \ref{Hbound1}.
The validity of \ref{Hbound2} is clear by the last assertion in (\ref{propl}).

\medskip

{\bf 7.} {\bf (Case $\mathbf{\displaystyle \frac{\pmb{\beta}}{2} \leq
    \frac{S}{S+2J}}$, viability of assumptions in step 5  and
  ${\pmb{\mathcal{C}}}^{\mathbf{1}}$ convergence)} 
We now examine (\ref{reqi3}), (\ref{reqi5}), (\ref{reqi6}). Under the usual
assumption $a-1, \gamma\ll 1$, these follow from:
\begin{equation}\label{faras2}
C\big(1+\|A\|_{0,\beta}\big)l_0^\beta\leq \|\mathcal{D}_0\|_0,
\qquad 2(1+\|\nabla v_0\|_0) \leq l_0b^{S\beta/6}.
\end{equation}
Hence we define:
$$l_0^\beta=\frac{\|\mathcal{D}_0\|_0}{C\big(1+\|A\|_{0,\beta}\big)},
\qquad b^{S\beta/6}= \frac{2(1+\|\nabla v_0\|_0)}{l_0}= 
\frac{C(1+\|\nabla v_0\|_0)\big(1+\|A\|_{0,\beta}\big)^{1/\beta}}{\|\mathcal{D}_0\|_0^{1/\beta}}$$
Consequently, the right hand side of the bound in \ref{propl4_1}
becomes, for $\gamma\ll 1$:
\begin{equation*}
\begin{split}
C\frac{b^{(S+2J)\gamma}}{l_0^{2a\gamma}} &\big(1+\|\nabla v_0\|_0\big)\|\mathcal{D}_0\|_0^{1/2}
\\ & = C\bigg(\frac{(1+\|\nabla v_0\|_0)^{\frac{6}{S\beta}(S+2J)} 
\big(1+\|A\|_{0,\beta}\big)^{\frac{6}{S\beta^2}(S+2J)+\frac{2a}{\beta}}}
{\|\mathcal{D}_0\|_0^{\frac{6}{S\beta^2}(S+2J)+\frac{2a}{\beta}}}\bigg)^\gamma  
\big(1+\|\nabla v_0\|_0\big)\|\mathcal{D}_0\|_0^{1/2}
\\ & \leq C
\big(1+\|A\|_{0,\beta}\big)^{\big(\frac{6}{S\beta^2}(S+2J)+\frac{2a}{\beta}\big)\gamma}\big(1+\|\nabla
v_0\|_0\big)^2 \|\mathcal{D}_0\|_0^{1/2 - \big(\frac{6}{S\beta^2}(S+2J)+\frac{2a}{\beta}\big)\gamma}
\end{split}
\end{equation*} 
As in the previous step, taking $\gamma$ small enough to have
$\big(\frac{6}{S\beta^2}(S+2J)+\frac{2a}{\beta}\big)\gamma\leq\frac{1}{8}$,
the above is further estimated by:
$$C \big(1+\|A\|_{0,\beta}\big)^{1/8}\big(1+\|\nabla v_0\|_0\big)^2 \|\mathcal{D}_0\|_0^{3/8},$$
eventually leading to (\ref{kalb}), as well as the same
convergences conclusion with \ref{Hbound1}, \ref{Hbound2}. 

\medskip

{\bf 8.} {\bf (Convergence in $\mathbf{\pmb{\mathcal{C}}^{1,\pmb{\alpha}}}$)} 
To conclude the proof, it remains to show the improved regularity of the limiting fields, namely: $\tilde
v\in\mathcal{C}^{1,\alpha}(\bar\omega,\R^k)$, $\tilde
w\in\mathcal{C}^{1,\alpha}(\bar\omega, \R^d)$. 
The interpolation inequality $\|\cdot\|_{1,\alpha} \leq
C\|\cdot\|_{1}^{1-\alpha} \|\cdot\|_{2}^{\alpha}$ and (\ref{Abd2}),
imply that for all $i\geq 0$ there holds on $\bar\omega$:
\begin{equation}\label{qitta}
\begin{split}
\|v_{i+2} - &v_{i+1}\|_{1,\alpha} + \|w_{i+2} - w_{i+1}\|_{1,\alpha}
\\ &  \leq C b^{\alpha J + \gamma} l_{i+1}^{1-\alpha -\alpha J (a-1) -a\gamma}
 M_{i+1} \frac{b^{(S+2J)\gamma}}{l_0^{2a\gamma}}\big(1+\|\nabla v_0\|_0\big)
\\ & = C \frac{b^{\alpha J + (S+2J+1)\gamma}}{l_0^{2a\gamma}}
\Big(B^{\frac{q^{i+1}-1}{q-1}} l_0^{q^{i+1}}\Big)^{1-\alpha - \alpha J (a-1) -a\gamma}
M_{i+1} \big(1+\|\nabla v_0\|_0\big)
\end{split}
\end{equation}
We will argue that the sequences $\{v_i\}_{i=0}^\infty$, $\{w_i\}_{i=0}^\infty$
are Cauchy in $\mathcal{C}^{1,\alpha}(\bar\omega)$, by comparing the right hand
side above with terms of a converging power series.

\medskip

\noindent Recall that the first case in step 4 is determined by
$\frac{\beta}{2}>\frac{S}{S+2J}$, which directly implies that:
$$0<\alpha<\frac{S}{S+2J}.$$ 
There, the quantities $M_{i+1}$ are given according to (\ref{def2}). 
We gather only those terms from the right hand side of (\ref{qitta})
that involve the counter $i$, so that $\|v_{i+2} -
v_{i+1}\|_{1,\alpha} + \|w_{i+2} - w_{i+1}\|_{1,\alpha}$  is bounded,  up to
a multiplier independent of $i$, by:
\begin{equation}\label{farfour}
\begin{split}
 \Big(&B^{\frac{q^{i+1}-1}{q-1}}l_0^{q^{i+1}-1}\Big)^{1-\alpha -\alpha J(a-1) -a\gamma}
\bigg(\frac{C(1+\|\nabla v_0\|_0)^2}{b^{S-\gamma} B^{\frac{S(a-1)-a\gamma}{q-1}}}\bigg)^{\frac{i+1}{2}}
\frac{1}{\big(B^{\frac{1}{q-1}} l_0\big)^{(q^{i+1}-1)\frac{J (a-1) +3a\gamma}{q-1}}}\\
&\quad = \Big(B^{\frac{1}{q-1}} l_0\Big)^{(q^{i+1}-1)\big(1-\frac{J(a-1)+3a\gamma}{q-1}-\alpha-\alpha
  J(a-1) -a\gamma\big)}
\bigg(\frac{C(1+\|\nabla v_0\|_0)^2}{b^{S-\gamma} B^{\frac{S(a-1)-a\gamma}{q-1}}}\bigg)^{\frac{i+1}{2}}
\\ & \quad \leq \Bigg(C B^{\frac{S-\frac{a}{a-1}\gamma}{S+2J+\frac{5a}{a-1}\gamma}-\alpha -\alpha J(a-a)- a\gamma}
\frac{1+\|\nabla v_0\|_0}{\big(b^{S-\gamma} B^{\frac{S(a-1)-a\gamma}{q-1}}\big)^{1/2}}\Bigg)^{i+1},
\end{split}
\end{equation}
where we used that $q^{i+1}-1\ge (q-1)(i+1)$. Observing now the calculation: 
$$b^{S-\gamma} B^{\frac{S(a-1)-a\gamma}{q-1}} = C
b^{S- (S-\frac{a}{a-1}\gamma)\frac{S+2J+(2S+4J+1)\gamma}{S+2J+\frac{5a\gamma}{a-1}}
  -\gamma} = C b^{\mathcal{O}(\gamma)},$$ 
and denoting $\delta = \frac{S}{S+2J}-\alpha>0$, we hence see that for $(a-1),\gamma\ll 1$,
the right hand side of (\ref{farfour}) is further bounded by:
$$ \bigg(C B^{\delta/3} \frac{1+\|\nabla v_0\|_0}{\big(b^{S-\gamma}
  B^{\frac{S(a-1)-a\gamma}{q-1}}\big)^{1/2}}\bigg)^{i+1} \leq 
\bigg(\frac{C (1+\|\nabla v_0\|_0)}{b^{\frac{\delta}{2}\big(\frac{S}{2}+J+(S+2J+\frac{1}{2})\gamma\big)}}\bigg)^{i+1}. $$
Consequently, the asserted comparison with the converging power series
is achieved provided that the ratio above is less than $1$, which is implied by:
$$b^{\delta S/4}\geq C (1+\|\nabla v_0\|_0),$$
and which is consistent with the defining requirements for $b, l_0$ in (\ref{faras}).

\medskip

\noindent The second case, in step 5, is determined by $\frac{\beta}{2}\leq
\frac{S}{S+2J}$, which implies:
$$0<\alpha<\frac{\beta}{2},$$
There, the quantities $M_{i+1}$ are given according to (\ref{def1}), so
the terms in the right hand side of (\ref{qitta}) that involve the counter $i$, are:
\begin{equation*}
\begin{split}
 \Big(&B^{\frac{q^{i+1}}{q-1}}l_0^{q^{i+1}}\Big)^{1-\alpha -\alpha J(a-1) -a\gamma}
\Big(\frac{1+\|\nabla v_0\|_0}{l_0}\Big)^{i} \frac{1}{\big(B^{\frac{1}{q-1}} l_0\big)^{q^{i}(q-\frac{\beta}{2})}}\\
&\quad = \Big(B^{\frac{1}{q-1}}l_0\Big)^{q^i\big(\frac{\beta}{2}- q\alpha -q\alpha J(a-1) -qa\gamma\big)}
\Big(\frac{1+\|\nabla v_0\|_0}{l_0}\Big)^{i}. 
\end{split}
\end{equation*} 
Using the bound $q^i\ge (q-1)i+1$ valid for all $i\geq 0$, and denoting
$\delta=\frac{\beta}{2}-\alpha>0$, we estimate the above displayed
quantity, up to a multiplier independent of $i$, by:
\begin{equation*}%\label{farfour2}
\begin{split}
\Bigg(\frac{B^{\frac{\beta}{2}- q\alpha -q\alpha J(a-1) -qa\gamma}(1+\|\nabla v_0\|_0)}{l_0}\Bigg)^{i}
\leq  \bigg(\frac{B^{\delta/2}(1+\|\nabla v_0\|_0)}{l_0}\bigg)^{i}
\leq \bigg(\frac{1+\|\nabla
  v_0\|_0}{l_0b^{\frac{\delta}{2}\big(\frac{S}{2}+J+(S+2J+\frac{1}{2})\gamma\big)}}\bigg)^{i}.
\end{split}
\end{equation*}
We see that the ratio of the related power series is less that $1$ provided that:
$$l_0b^{\delta S/4}>1+\|\nabla v_0\|_0,$$
which is consistent with the requirements in (\ref{faras2}) in step 7.
This ends the proof of the $\mathcal{C}^{1,\alpha}$ convergences and
completes the proof of Theorem \ref{th_NashKuiHol}.
\endproof

\section{A proof of Theorem \ref{th_final}}\label{sec5}

We first replace $\omega$ by its smooth superset, on which $v,w,A$ are
defined and (\ref{baqara}) holds. Without loss of generality, the same
is true on its closed $2l$-neighbourhood $\bar\omega+\bar B_{2l}(0)$,
for some $0<l<l_0$ that allows for the application of
Corollary \ref{th_NKH}. Fix $\epsilon\ll 1$, small as indicated below.
First, we let $v_1\in\mathcal{C}^\infty(\bar\omega +\bar
B_l(0),\R^k)$, $w_1\in\mathcal{C}^\infty(\bar\omega +\bar B_l(0),\R^2)$,
$A_1\in \mathcal{C}^\infty(\bar\omega +\bar B_l(0),\R^{2\times 2}_{\sym})$ with:
\begin{equation*}
\begin{split}
&\|v_1- v\|_1\leq \epsilon^5,\quad \|w_1- w\|_1\leq \epsilon^5,\quad
\|A_1- A\|_0\leq \epsilon^5,\\
&\mathcal{D}_1= A_1 -\big(\frac{1}{2}(\nabla v_1)^T\nabla v_1 + \sym\nabla w_1\big)
> c_1\Id_2 \quad\mbox{ on $\bar\omega +\bar B_l(0)\;$ for some } c_1>0.
\end{split}
\end{equation*}
The last property above follows from:
\begin{equation}\label{osta}
\begin{split}
\|\mathcal{D}_1-\mathcal{D}\|_0& \leq \|A_1-A\|_0 + \|\nabla (w_1-w)\|_0
+\frac{1}{2}\|\nabla (v_1-v)\|_0 \big(\|\nabla v_1\|_0+\|\nabla v\|_0\big)
\\ & \leq 3\epsilon^5 (1+\|\nabla v\|_0).
\end{split}
\end{equation}
Second, use Lemma \ref{th_approx_nonlocal} to get
$v_2\in\mathcal{C}^\infty(\bar\omega +\bar B_l(0),\R^k)$,
$w_2\in\mathcal{C}^\infty(\bar\omega +\bar B_l(0),\R^2)$ satisfying:
\begin{equation*}
\begin{split}
&\|v_2- v_1\|_0\leq \epsilon^5,\quad \|w_2- w_1\|_0\leq \epsilon^5,\\
& \|\nabla (v_2-v_1)\|_0\leq C\|\mathcal{D}_1\|_0^{1/2}\leq
C\big(\|\mathcal{D}\|_0^{1/2} + \epsilon^{5/2} + \|\nabla v\|_0^{1/2}\big),\\
&\mathcal{D}_2= A_1 -\big(\frac{1}{2}(\nabla v_2)^T\nabla v_2 + \sym\nabla w_2\big)
\quad\mbox{ satisfies }\; \|\mathcal{D}_2\|_0\leq \epsilon^5,
\end{split}
\end{equation*}
where we applied (\ref{osta}) in the gradient increment bound of $v$.

\smallskip

If the deficit $\mathcal{D}_3$, defined on $\bar\omega +\bar B_l(0)$ in:
$$\mathcal{D}_3 = A- \big(\frac{1}{2}(\nabla v_2)^T\nabla v_2 + \sym\nabla w_2\big) $$
is equivalently zero then we may simply
take $\tilde v= v_2$ and $\tilde w=w_2$ to satisfy the claim of the
Theorem. Otherwise, we use Corollary \ref{th_NKH} to get $v_2$, $w_2$ and $A$, since:
$$0<\|\mathcal{D}_3\|_0 \leq \|A-A_1\|_0 + \|\mathcal{D}_2\|_0\leq 2\epsilon^5\leq 1,$$
and consequently obtain $\tilde v\in\mathcal{C}^{1,\alpha}(\bar\omega,\R^k)$,
$\tilde w\in\mathcal{C}^{1,\alpha}(\bar\omega,\R^2)$ with the properties:
\begin{equation*}
\begin{split}
&\|\tilde v - v_2\|_0\leq C(1+\|\nabla v_2\|_0)^2\|\mathcal{D}_3\|_0^{1/4}
\leq C\big(1+\|\nabla v_0\|_0+\|\mathcal{D}\|_0\big)\epsilon^{5/4},\\ &
\|\tilde w- w_2\|_0\leq C(1+\|\nabla v_2\|_0)^3\|\mathcal{D}_3\|_0^{1/4}
\leq C\big(1+\|\nabla v_0\|_0^{3/2}+\|\mathcal{D}\|_0^{3/2}\big)\epsilon^{5/4},
\\ & A -\big(\frac{1}{2}(\nabla \tilde v)^T\nabla \tilde v + \sym\nabla
\tilde w\big) = 0 \quad\mbox{ in }\; \bar\omega.
\end{split}
\end{equation*} 
It now suffices to take $\epsilon$ sufficiently small (in function of
$\|\mathcal{D}\|_0$, $\|\nabla v\|_0$ and of $C$
that depend only on $\omega, k, A$ and $\alpha$), to replace the right
hand sides of both bounds above have $\epsilon^{6/5}$. Thus:
\begin{equation*}
\begin{split}
&\|\tilde v - v\|_0\leq \|\tilde v - v_2\|_0 + \| v_2 - v_1\|_0 + \|v_1 - v\|_0\leq
3\epsilon^{6/5} \leq \epsilon,\\ &
\|\tilde w - w\|_0\leq \|\tilde w - w_2\|_0 + \| w_2 - w_1\|_0 + \|w_1 - w\|_0
\leq 3\epsilon^{6/5}\leq \epsilon,
\end{split}
\end{equation*} 
for $\epsilon\ll 1$. The proof is done.
\endproof

\section{Application: energy scaling bound for thin films}\label{sec_appli}

In this section, we present an application of Theorem
\ref{th_final} towards obtaining a energy bound on a multidimensional non-Euclidean
elasticity functional on thin films with two-dimensional midplate. 
More precisely, given $\omega\subset\R^2$ we consider the
family of  domains:
$$\Omega^h = \big\{(x,z); ~x\in\omega, ~ z\in B(0,h)\subset\R^k\big\}\subset\R^{2+k},$$
parametrised by $h\ll 1$ and the family of Riemannian metrics on $\Omega^h$ of the form:
$$g^h=\Id_{2+k}+2h^{\gamma/2}S, \quad\mbox{ where } \gamma>0 \mbox{
  and } S\in\mathcal{C}^\infty(\bar\omega,\R^{(2+k)\times (2+k)}_\sym).$$
We then pose the problem of minimizing the following functionals, as  $h\to 0$:
$$\mathcal{E}^h(u) = \fint_{\Omega^h} W\big((\nabla u)(g^h)^{-1/2}\big)~\mathrm{d}(x,z)
\qquad\mbox{ for all }\; u\in H^1(\Omega^h,\R^{2+k}).$$
The density function $W:\R^{(d+k)\times (d+k)}\to [0,\infty]$ is assumed
to be $\mathcal{C}^2$-regular in the vicinity of the special
orthogonal group of rotations $\mathrm{SO}(2+k)$, to be equal to
$0$ at $\Id_{2+k}$, and to be frame-invariant in the sense that $W(RF)= W(F)$
for all $R\in \mathrm{SO}(2+k)$. The value $\mathcal{E}^h(u)$ may be
interpreted as the averaged pointwise deficit of $u$ from being the
orientation preserving isometric immersion of $g^h$ on
$\Omega^h$. When $k=1$ then $\mathcal{E}^h(u)$ models
the elastic energy (per unit thickness)  of the deformation $u$ of a thin
three-dimensional film with midplate $\omega$ and thickness $2h$ and
prestrained by $g^h$. Questions on the asymptotics
of minimizing configurations to $\mathcal{E}^h$ as
$h\to 0$, in function of the scaling exponent $\beta$ in: $\inf \mathcal{E}^h\sim
Ch^\beta$, received a lot of attention, particularly via techniques of dimension reduction and
$\Gamma$-convergence, starting with the seminal paper \cite{FJM}, see also
\cite{lew_book} and references therein. 
 
\medskip

\noindent Extending the analysis for $d=2$, $k=1$ in \cite[Theorem
1.4]{JBL} and following verbatim the general proof of \cite[Theorem
7.1]{lew_conv} (valid with arbitrary $d,k\geq 1$), we get:

\begin{theorem}\label{th_scaling}
Assume that $\omega\subset\R^2$ is an open, bounded domain and let
$k\geq 1$. Denote $s = \frac{4}{k}$. Then, there holds:
\begin{itemize}
\item[(i)] if $\gamma\ge 4$, then $\inf \mathcal{E}^h\leq Ch^{\beta}$, for
  every $\beta<2+\frac{\gamma}{2}$,

\item[(ii)] if $\gamma\in \big[\frac{4k}{3k+4},4\big)$, then
  $\inf\mathcal{E}^h\leq Ch^{\beta}$ for every $\beta< \frac{4k+\gamma (k+4)}{2k+4}$,

\item[(iii)] if $\gamma\in \big(0,\frac{4k}{3k+4}\big)$, then $\inf \mathcal{E}^h\leq Ch^{2\gamma}$.
\end{itemize}
\end{theorem}

\end{document}